%% file: Paper_Adaptive.tex
\documentclass{article}

\usepackage{times}
\input{commands}
\usepackage{graphicx} 
\usepackage{subfigure} 
\usepackage[colorinlistoftodos,bordercolor=orange,backgroundcolor=orange!20,linecolor=orange,textsize=scriptsize]{todonotes}

\usepackage{natbib}

\usepackage{algorithm}
\usepackage{algorithmic}
\usepackage{amssymb}
\usepackage{amsmath, amsthm}
\usepackage{epstopdf}

\usepackage{hyperref}

\usepackage[accepted]{icml2015}

\icmltitlerunning{Stochastic Dual Coordinate Ascent with Adaptive Probabilities}

\begin{document} 

\twocolumn
[
\icmltitle{Stochastic Dual Coordinate Ascent with Adaptive Probabilities}

\icmlauthor{Dominik Csiba}{cdominik@gmail.com}
\icmlauthor{Zheng Qu}{zheng.qu@ed.ac.uk}
\icmlauthor{Peter Richt\'{a}rik}{peter.richtarik@ed.ac.uk}
\icmladdress{University of Edinburgh}

\icmlkeywords{adaptivity, stochastic dual coordinate ascent, empirical risk minimization, optimization}

\vskip 0.3in
]

\begin{abstract} 

This paper introduces AdaSDCA: an adaptive variant of  stochastic dual coordinate ascent (SDCA) for solving the regularized empirical risk minimization  problems. Our modification consists in allowing the method adaptively change the probability distribution over the dual variables throughout the iterative process. AdaSDCA achieves provably better  complexity bound than SDCA with the best fixed probability distribution, known as importance sampling. However, it is of a theoretical character as it is expensive to implement. We also propose AdaSDCA+: a practical variant  which in our experiments outperforms existing non-adaptive methods.
\end{abstract} 

\section{Introduction}
\label{introduction}

\textbf{Empirical Loss Minimization.} In this paper we consider the regularized empirical risk minimization problem:
\begin{equation} \label{def:primal}
\min_{w \in \R^d} \left[ P(w) \eqdef \frac{1}{n}\sum_{i=1}^n \phi_i(A_i^\top w) + \lambda g(w) \right].
\end{equation}
In the context of supervised learning, $w$ is a linear predictor, $A_1,\dots, A_n\in \R^d$ are samples, $\phi_1,\dots,\phi_n:\R^d\to \R$ are loss functions, $g:\R^d\to \R$ is a regularizer and $\lambda>0$ a regularization parameter. Hence, we are seeking to identify the predictor which minimizes the average (empirical) loss $P(w)$.

We assume throughout that the loss functions are $1/\gamma$-smooth for some $\gamma>0$. That is, we assume they are differentiable and have Lipschitz derivative with Lipschitz constant $1/\gamma$: \[|\phi'(a)-\phi'(b)|\leq \frac{1}{\gamma}|a-b|\] for all $a,b\in \R$. Moreover, we assume that $g$ is $1$-strongly convex with respect to the L2 norm:
\[g(w) \leq \alpha g(w_1) + (1-\alpha)g(w_2) - \frac{\alpha(1-\alpha)}{2}\|w_1-w_2\|^2\]
for all $w_1,w_2\in \dom g$, $0\leq \alpha\leq 1$ and $w=\alpha w_1 +(1-\alpha) w_2$.

The ERM problem~\eqref{def:primal} has received considerable attention in recent years due to its widespread usage in supervised statistical learning \cite{SDCA}. Often, the number of samples $n$ is  very large and it is important to design algorithms that would be efficient in this regime.  

\textbf{Modern stochastic algorithms for ERM.} Several highly efficient methods for solving the ERM problem were proposed and analyzed recently. These include primal methods such as SAG \cite{SAG}, SVRG \cite{SVRG}, S2GD \cite{S2GD},  SAGA \cite{SAGA}, mS2GD \cite{mS2GD} and MISO \cite{MISO}. Importance sampling was considered in  ProxSVRG \cite{proxSVRG} and S2CD  \cite{S2CD}. 

\textbf{Stochastic Dual Coordinate Ascent.} One of the most successful methods in this category is  {\em stochastic dual coordinate ascent (SDCA)}, which  operates on the dual of the ERM problem \eqref{def:primal}:
\begin{equation} \label{def:dual}
\max_{\alpha = (\alpha_1, \dots, \alpha_n) \in \R^n} \left[ D(\alpha) \eqdef - f(\alpha) - \psi(\alpha) \right],
\end{equation}
where functions $f$ and $\psi$ are  defined by  
\begin{align}
f(\alpha) &\eqdef \lambda g^*\left( \frac{1}{\lambda n}\sum_{i=1}^n A_i \alpha_i \right), \label{def:f}
\\ \psi(\alpha) &\eqdef \frac{1}{n} \sum_{i=1}^n \phi_i^*(-\alpha_i), \label{def:psi}
\end{align}
and $g^*$ and $\phi_i^*$ are the convex conjugates\footnote{By the convex (Fenchel) conjugate of a function $h : \R^k \rightarrow \R$ we mean the function $h^* : \R^k \rightarrow \R$ defined by $h^*(u)~=~\sup_s\{s^\top u - h(s)\}$.} of $g$ and $\phi_i$, respectively. Note that in dual problem, there are as many variables as there are samples in the primal: $\alpha\in \R^n$. 

SDCA in each iteration randomly selects a  dual variable $\alpha_i$, and performs its update, usually via  closed-form formula -- this strategy is know as randomized coordinate descent. Methods based on updating randomly selected dual variables  enjoy, in our setting, a linear convergence rate~\cite{SDCA, ProxSDCA, pegasos2, ASDCA, IProx-SDCA, Quartz}. These methods have attracted considerable attention in the past few years, and include SCD \cite{ShalevTewari09}, RCDM \cite{Nesterov:2010RCDM}, UCDC \cite{UCDC}, ICD   \cite{ICD}, PCDM \cite{PCDM}, SPCDM \cite{SPCDM}, SPDC~\cite{SPDC}, APCG~\cite{APCG}, RCD \cite{Necoara:rcdm-coupled},  APPROX \cite{APPROX}, QUARTZ  \cite{Quartz} and ALPHA \cite{ALPHA}. Recent advances on mini-batch and distributed variants can be found in~\cite{WrightAsynchrous14},~\cite{AccMiniBatchRBCDNIPS14},~\cite{Hydra},~\cite{Hydra2},~\cite{TrofimovGenkin14},~\cite{CoCoa}, ~\cite{DisBCDM} and~\cite{MahajanKeerthiSunda}. Other related work includes~\cite{nemirovski2009robust, AdaGrad, AgarwalBottou14, Pathwisecd,RBCDFT, DQA}.  We also point to~\cite{ReviewWright} for a review on coordinate descent algorithms.

\textbf{Selection Probabilities.} Naturally, both the theoretical convergence rate and practical performance of randomized coordinate descent methods depends on  the probability distribution governing the choice of individual coordinates. While most existing work assumes uniform distribution, it was shown by \citet{UCDC, Necoara:Coupled, IProx-SDCA}  that  coordinate descent works for an arbitrary fixed probability distribution over individual coordinates and even subsets of coordinates \cite{NSync, Quartz, ALPHA, ESO}. In all of these works the theory allows the computation of a fixed probability distribution, known as {\em importance sampling}, which optimizes the complexity bounds. However, such a distribution often depends on unknown quantities, such as the distances of the individual variables from their optimal values \cite{UCDC, ALPHA}. In some cases, such as for smooth strongly convex functions or in the primal-dual setup we consider here, the probabilities forming an importance sampling can be explicitly computed \cite{NSync, IProx-SDCA, Quartz, ALPHA, ESO}. Typically, the theoretical influence of using the importance sampling is in the replacement of the maximum of certain data-dependent quantities in the complexity bound by the average.


\textbf{Adaptivity.} Despite the striking developments in the field, there is virtually no literature on  methods using an {\em adaptive} choice of the probabilities. We are aware of a few pieces of work; but all  resort to heuristics unsupported by theory~\cite{ glasmachers2013accelerated, randomFrankWolfe, Nopesky,NewClassesCD, Adaptivecd}, which unfortunately also means that the methods are sometimes effective, and sometimes not.  {\bf We observe that in the  primal-dual framework we consider, each dual variable can be equipped with a natural measure of progress which we call ``dual residue''. We propose that the selection probabilities be constructed based on these quantities.} 


{\em Outline:} In Section~\ref{s:prob} we summarize the contributions of our work. In Section~\ref{sec:algo} we describe our first, theoretical methods (Algorithm~\ref{alg:A}) and describe the intuition behind it. In Section~\ref{s:conv} we provide  convergence analysis. In Section \ref{s:heuristic} we introduce Algorithm \ref{alg:HA}: an  variant of Algorithm \ref{alg:A} containing heuristic elements which make it efficiently implementable. We conclude with numerical experiments in Section \ref{s:Experiments}. Technical proofs and additional numerical experiments can be found in the appendix.

\section{Contributions} \label{s:prob}

We now briefly highlight the main contributions of this work. 

\textbf{Two algorithms with adaptive probabilities.} We propose two new stochastic dual ascent algorithms: AdaSDCA (Algorithm~\ref{alg:A}) and AdaSDCA+ (Algorithm~\ref{alg:HA}) for solving~\eqref{def:primal} and its dual problem~\eqref{def:dual}. The novelty of our algorithms is in adaptive choice of the probability distribution over the dual coordinates. 

\textbf{Complexity analysis.} We provide a convergence rate analysis for the first method, showing that {\bf AdaSDCA enjoys better rate than the best known rate for SDCA with a fixed   sampling}~\cite{IProx-SDCA, Quartz}. The probabilities are proportional to a certain measure of dual suboptimality associated with each variable.

\textbf{Practical method.} AdaSDCA requires the same computational effort per iteration as the batch gradient algorithm. To solve this issue, we propose AdaSDCA+ (Algorithm~\ref{alg:HA}): an efficient heuristic variant of the AdaSDCA. The computational effort of the heuristic  method in a single iteration is low, which makes it very competitive with  methods based on importance sampling, such as IProx-SDCA~\cite{IProx-SDCA}. We support this with computational experiments in Section~\ref{s:Experiments}.

{\em Outline:} In Section~\ref{s:prob} we summarize the contributions of our work. In Section~\ref{sec:algo} we describe our first, theoretical methods (AdaSDCA) and describe the intuition behind it. In Section~\ref{s:conv} we provide  convergence analysis. In Section \ref{s:heuristic} we introduce AdaSDCA+: a variant of AdaSDCA containing heuristic elements which make it efficiently implementable. We conclude with numerical experiments in Section \ref{s:Experiments}. Technical proofs and additional numerical experiments can be found in the appendix.



\section{The Algorithm: AdaSDCA}\label{sec:algo}

It is well known that the optimal primal-dual pair $(w^*,\alpha^*)\in \R^d\times \R^n$ satisfies the following \textit{optimality conditions}:
\begin{align} \label{eq:optimality}
w^* &= \nabla g^*\left(\frac{1}{\lambda n}A\alpha^*\right)
\\ \alpha^*_i &= -\nabla\phi_i(A_i^\top w^*),\enspace \forall i\in [n]\eqdef \{1, \dots, n\},
\end{align} 
where $A$ is the $d$-by-$n$ matrix with columns $A_1,\dots, A_n$.

\begin{definition}[Dual residue] The \textit{dual residue}, $\kappa=(\kappa_1,\dots,\kappa_n)\in \R^ n$, associated with $(w,\alpha)$  is given by:
 \begin{equation}\label{def:kappa}
\kappa_i \eqdef \alpha_i + \nabla\phi_i(A_i^\top w).
 \end{equation}
 \end{definition}
 
Note, that $\kappa_i^t = 0$ if and only if $\alpha_i$ satisfies \eqref{eq:optimality}. This motivates the design of AdaSDCA (Algorithm~\ref{alg:A}) as follows: whenever  $|\kappa_i^t|$ is large, the $i$th dual coordinate $\alpha_i$ is suboptimal and hence should be updated more often.

\begin{definition}[Coherence] We say that  probability vector $p^t\in \R^n$ is {\em coherent} with the dual residue $\kappa^t$ if for all $i\in [n]$ we have
\[\kappa_i^t\neq 0 \quad \Rightarrow \quad p_i^t >0.\]
\end{definition}

Alternatively, $p^t$ is coherent with $k^t$ if for 
$$
I_t\eqdef \{i\in [n]: \kappa_i^t \neq 0\}\subseteq [n].
$$
we have $\min_{i \in I_t}p_i^t>0$.

\begin{algorithm}
\caption{AdaSDCA} \label{alg:A}
\begin{algorithmic}

\STATE \textbf{Init:} $v_i=A_i^\top A_i$ for  $i\in [n]$; $\alpha^0 \in \R^n$; $\bar \alpha^ 0=\frac{1}{\lambda n} A \alpha^ 0$
\FOR{$t \geq 0$}
\STATE Primal update: $w^t =  \nabla g^* \left(\bar\alpha^{t}\right)$
\STATE Set: $\alpha^{t+1}=\alpha^{t}$
\STATE Compute residue $\kappa^t$: $\kappa_i^t=\alpha_i^t+\nabla \phi_i(A_i^ \top w^ t), \forall i\in [n]$
\STATE Compute probability distribution $p^t$ coherent with $\kappa^t$ 
\STATE Generate random $i_t \in [n]$ according to $p^t$
\STATE Compute: \\
 $\qquad \Delta\alpha_{i_t}^t = \displaystyle\argmax_{\Delta \in \R} \left\{-\phi_{i_t}^*(-(\alpha_{i_t}^{t}+\Delta))\right.$\\
$\qquad \qquad\qquad\qquad\quad\left.-A_{i_t}^\top w^t\Delta-
\frac{v_{i_t}}{2\lambda n}|\Delta|^2\right\}$
\STATE Dual update: $\alpha_{i_t}^{t+1} = \alpha_{i_t}^{t} + \Delta\alpha_{i_t}^t$
\STATE Average update: $\bar \alpha^t=\bar \alpha^t+\frac{\Delta\alpha_{i_t}}{\lambda n}A_{i_t}$
\ENDFOR
\STATE \textbf{Output:} $w^t, \alpha^t$
\end{algorithmic}
\end{algorithm}

AdaSDCA is a stochastic dual coordinate ascent method, with an adaptive probability vector $p^t$, which  could potentially change at every iteration $t$. The primal and dual update rules are exactly the same as in  standard SDCA~\cite{SDCA}, which instead uses uniform sampling probability at every iteration  and does not require the computation of the dual residue $\kappa$.

Our first result highlights a key technical tool which ultimately leads to the development of good adaptive sampling distributions $p^t$ in AdaSDCA. For simplicity we denote by $\Exp_t$ the expectation with respect to the random index $i_t\in [n]$ generated at iteration $t$.

\begin{lemma}\label{l:Theoutof}
Consider the AdaSDCA algorithm during iteration $t\geq 0$ and assume that $p^t$ is coherent with $\kappa^t$. Then
 \begin{align}
&\Exp_t\left[D(\alpha^{t+1})-D(\alpha^{t})\right]- \theta \left(P(w^t)-D(\alpha^t)\right) \notag\\&
\geq -\frac{\theta}{2\lambda n^2}\sum_{i\in I_t}\left(  \frac{\theta(v_i+n\lambda\gamma)}{p_i^t} - n\lambda \gamma \right) |\kappa^t_i|^2 \label{a:Theoutof}
,\end{align}
for arbitrary \begin{align}\label{a:thetaleqmin}0\leq \theta \leq  \min_{i\in I_t} p^t_i.\end{align}
\end{lemma}
\begin{proof}
Lemma~\ref{l:Theoutof} is proved similarly to Lemma 2 in~\cite{IProx-SDCA}, but in a slightly more general setting. For completeness, we provide the proof in the appendix.
\end{proof}

Lemma~\ref{l:Theoutof} plays a key role in the analysis of stochastic dual coordinate methods~\cite{SDCA, IProx-SDCA, ASDCA}. 
Indeed, if the right-hand side of~\eqref{a:Theoutof} is positive, then the 
primal dual error  $P(w^t)-D(\alpha^t)$ can be bounded 
by the expected dual ascent $\Exp_t[D(\alpha^{t+1})-D(\alpha^t)]$ times $1/\theta$, which yields the contraction of the dual error at the rate 
of $1-\theta$ (see Theorem~\ref{th:general}). In order to make the right-hand side of~\eqref{a:Theoutof} positive we can
take any $\theta$ smaller than $\theta(\kappa^t,p^t)$ where the function $\theta(\cdot,\cdot):\R_{+}^n\times \R_+^n\rightarrow \R$ is defined by:
\begin{align}
\theta(\kappa,p)\equiv
\frac{n\lambda \gamma \sum_{i:\kappa_i\neq 0} |\kappa_i|^2}{\sum_{i: \kappa_i\neq 0} p_i^{-1} |\kappa_i|^2(v_i+n\lambda \gamma)}
.\end{align}
We also need to make sure that $0\leq \theta\leq \min_{i\in I_t} p_i^t$ in order to apply Lemma~\ref{l:Theoutof}.
A ``good'' adaptive probability $p^t$ should then be the solution of the following optimization problem:

\begin{eqnarray}\label{eq:optp}
 &\displaystyle\max_{p\in \R_+^ n}& \theta(\kappa^t,p)\\ \notag
\quad &\mathrm{s.t.}&  \sum_{i=1}^ n p_i =1\\ \notag & & \theta (\kappa^t, p)\leq \min_{i: \kappa^t_i\neq 0} p_i
\end{eqnarray}

A feasible solution to \eqref{eq:optp} is the \textit{importance sampling} (also known as optimal serial sampling) $p^*$ defined by:
\begin{align}\label{a:pstar}
p_i^*\eqdef\frac{v_i+n\lambda \gamma}{\sum_{j=1}^n\left(v_j+n\lambda \gamma\right)},\enspace \forall i\in [n],
\end{align}
which was proposed in~\cite{IProx-SDCA} to obtain proximal stochastic dual coordinate ascent method with importance sampling (IProx-SDCA). The same optimal probability vector was also deduced, via different means and in a more general setting in~\cite{Quartz}. Note that in this special case, since $p^t$ is independent of the residue $\kappa^t$,  the computation of $\kappa^t$ is unnecessary and hence the complexity of each iteration does not scale up with $n$.  

It seems difficult to identify other feasible solutions to program~\eqref{eq:optp} apart from $p^*$, not to mention solve it exactly.
However, by relaxing the constraint $\theta (\kappa^t, p)\leq \min_{i: \kappa^t_i\neq 0} p_i$, we obtain an explicit optimal 
solution.

\begin{lemma}\label{lem:pstar}
 The optimal solution $ p^*(\kappa^t)$ of 
\begin{eqnarray}\label{eq:optprelaxed}
 &\displaystyle\max_{p\in \R_+^ n}& \theta(\kappa^t,p)\\ \notag
\quad &\mathrm{s.t.}&  \sum_{i=1}^ n p_i =1
\end{eqnarray}
is:
\begin{align}\label{a:optipsdf}
 (p^*(\kappa^t))_i=\frac{|\kappa_i^t|\sqrt{v_i + n\lambda\gamma}}{\sum_{j=1}^n |\kappa_j^t| \sqrt{v_j + n\lambda\gamma}},\enspace \forall i\in [n].
\end{align}
\end{lemma}
\begin{proof}
The proof is deferred to the appendix. 
\end{proof}

The suggestion made by~\eqref{a:optipsdf} is clear: we should update
more often those dual coordinates $\alpha_i$ which have  large absolute dual residue  $|\kappa_i^t|$ and/or large Lipschitz constant $v_i$.

If we let $p^t=p^*(\kappa^t)$ and $\theta=\theta (\kappa^t, p^t)$,  the constraint~\eqref{a:thetaleqmin} may not be sastified, 
in which case~\eqref{a:Theoutof}
does not necessarily hold. 
However, as shown by the next lemma, the constraint~\eqref{a:thetaleqmin} 
is not required for obtaining~\eqref{a:Theoutof} when all the functions $\{\phi_i\}_i$ are quadratic.

\begin{lemma}\label{l:Theoutofquadratic}Suppose that all $\{\phi_i\}_i$ are quadratic. Let $t\geq 0$.
If  $\min_{i\in I_t} p_i^t>0$, then~\eqref{a:Theoutof} holds for any $\theta \in [0, +\infty)$.
\end{lemma}

The proof is deferred to Appendix.

\section{Convergence results} \label{s:conv}

In this section we present our theoretical complexity results for AdaSDCA. The main results are formulated in Theorem~\ref{th:general}, covering  the
general case, and in Theorem~\ref{th:quadratic} in the special case when $\{\phi_i\}_{i=1}^n$ are all quadratic.

\subsection{General loss functions}

We derive the convergence result from Lemma~\ref{l:Theoutof}.
\begin{proposition}\label{prop-d1}
Let $t\geq 0$. If $\min_{i\in I_t} p_i^t>0$ and $\theta(\kappa^t, p^t)\leq \min_{i\in I_t} p_i^t$, then \begin{align*}
\Exp_t\left[D(\alpha^{t+1})-D(\alpha^{t})\right]\geq  \theta(\kappa^t,p^t) \left(P(w^t)-D(\alpha^t)\right).\end{align*}
\end{proposition}
\begin{proof}
 This follows directly from Lemma~\ref{l:Theoutof} and the fact that the right-hand side of~\eqref{a:Theoutof} equals 0 when $\theta=\theta(\kappa^t, p^t)$.
\end{proof}

\begin{theorem}\label{th:general}
 Consider AdaSDCA. If at each iteration $t\geq 0$, $\min_{i\in I_t} p_i^t>0$ and $\theta(\kappa^t, p^t)\leq \min_{i\in I_t} p_i^t$, then 
\begin{align}\label{a:esdfdf}
\Exp[P(w^t)-D(\alpha^t)]
\leq \frac{1}{\tilde \theta_t}  \prod_{k=0}^t(1-\tilde \theta_k) \left(D(\alpha^*)-D(\alpha^0)\right),
\end{align}
for all $t\geq 0$
where
\begin{align}\label{a:tildetheta}
\tilde \theta_t\eqdef \frac{\Exp[\theta(\kappa^t, p^t)(P(w^t)-D(\alpha^t))]}{\Exp[P(w^t)-D(\alpha^t)]}.
\end{align}
\end{theorem}
\begin{proof}
By Proposition~\ref{prop-d1}, we know that 
\begin{align}\notag
\Exp[D(\alpha^{t+1})-D(\alpha^t)]
&\geq \Exp[\theta(\kappa^t, p^t)(P(w^t)-D(\alpha^t))]
\\& \overset{\eqref{a:tildetheta}}{=} \tilde \theta_t \Exp[P(w^t)-D(\alpha^t)]  \label{a:thetaPwa}
\\ & \geq \tilde \theta_t \Exp[D(\alpha^*)-D(\alpha^t)]\notag,
\end{align}
whence 
$$
\Exp[D(\alpha^*)-D(\alpha^{t+1})] \leq (1-\tilde \theta_t) \Exp[D(\alpha^*)-D(\alpha^t)].
$$
Therefore,
$$
\Exp[D(\alpha^*)-D(\alpha^t)]\leq \prod_{k=0}^t (1-\tilde \theta_k) \left(D(\alpha^*)-D(\alpha^0)\right).
$$
By plugging the last bound into~\eqref{a:thetaPwa} we get the bound on the primal dual error:
\begin{align*}
\Exp[P(w^t)-D(\alpha^t)]&\leq \frac{1}{\tilde \theta_t} \Exp[D(\alpha^{t+1})-D(\alpha^t)]\\& \leq \frac{1}{\tilde \theta_t} \Exp[D(\alpha^*)-D(\alpha^t)]
\\ & \leq \frac{1}{\tilde \theta_t}\prod_{k=0}^t (1-\tilde \theta_k) \left(D(\alpha^*)-D(\alpha^0)\right).\qedhere
\end{align*}

\end{proof}

As mentioned in Section~\ref{sec:algo}, by letting every sampling probability $p^t$ be the importance sampling (optimal serial sampling) $p^*$ defined in~\eqref{a:pstar}, AdaSDCA reduces to 
IProx-SDCA proposed in~\cite{IProx-SDCA}. The convergence theory established for IProx-SDCA in~\cite{IProx-SDCA}, which can also be derived as a direct corollary of our Theorem~\ref{th:general}, is stated as follows.

\begin{theorem}[\cite{IProx-SDCA}]
Consider AdaSDCA with $p^t=p^*$ defined in~\eqref{a:pstar} for all $t\geq 0$. Then 
\begin{align*}
\Exp[P(w^t)-D(\alpha^t)]
\leq \frac{1}{\theta_*}  (1-\theta_*)^t \left(D(\alpha^*)-D(\alpha^0)\right),
\end{align*}
where $$\theta_*=\frac{n\lambda \gamma}{\sum_{i=1}^n (v_i+\lambda \gamma n )}.$$
\end{theorem}

The next corollary suggests that a {\bf better convergence rate than IProx-SDCA can be achieved by using properly chosen adaptive sampling probability.} 

\begin{corollary}\label{coro:general}
Consider AdaSDCA.
If at each iteration $t\geq 0$, $p_t$ is the optimal solution of~\eqref{eq:optp}, then~\eqref{a:esdfdf} holds and $\tilde \theta_t \geq \theta_*$ for all $t\geq 0$.
\end{corollary}
However, solving~\eqref{eq:optp} requires large computational effort, because of the dimension $n$ and the non-convex structure of the program. 
 We show in the next section that when all the loss functions $\{\phi_i\}_i$ are quadratic, then we can get better convergence rate in theory than IProx-SDCA by using the optimal solution of~\eqref{eq:optprelaxed}.

\subsection{Quadratic loss functions}
The main difficulty of solving~\eqref{eq:optp} comes from the inequality constraint, which originates from~\eqref{a:thetaleqmin}.
In this section we mainly show that the constraint~\eqref{a:thetaleqmin} can be released if all $\{\phi_i\}_i$ are quadratic.

\begin{proposition}\label{prop-d2} Suppose that all $\{\phi_i\}_i$ are quadratic.
Let $t\geq 0$. If $\min_{i\in I_t} p_i^t>0$, then \begin{align*}
\Exp_t\left[D(\alpha^{t+1})-D(\alpha^{t})\right]\geq  \theta(\kappa^t,p^t) \left(P(w^t)-D(\alpha^t)\right).\end{align*}
\end{proposition}
\begin{proof}
This is a direct consequence of Lemma~\ref{l:Theoutofquadratic} and the fact that the right-hand side of~\eqref{a:Theoutof} equals 0 when $\theta=\theta(\kappa^t, p^t)$.
\end{proof}
\begin{theorem}\label{th:quadratic}Suppose that all $\{\phi_i\}_i$ are quadratic.
 Consider AdaSDCA. If at each iteration $t\geq 0$, $\min_{i\in I_t} p_i^t>0$, then~\eqref{a:esdfdf} holds for all $t\geq 0$.
 \end{theorem}
 \begin{proof}
 We only need to apply Proposition~\ref{prop-d2}. The rest of the proof is the same as in Theorem~\ref{th:general}.
 \end{proof}

\begin{corollary}\label{coro:quadratic}
Suppose that all $\{\phi_i\}_i$ are quadratic.
Consider AdaSDCA.
If at each iteration $t\geq 0$, $p_t$ is the optimal solution of~\eqref{eq:optprelaxed}, which has a closed form~\eqref{a:optipsdf}, then~\eqref{a:esdfdf} holds and $\tilde \theta_t \geq \theta_*$ for all $t\geq 0$.
\end{corollary}

\section{Efficient heuristic variant} \label{s:heuristic}

 Corollary~\ref{coro:general} and~\ref{coro:quadratic} suggest how to choose adaptive sampling probability in AdaSDCA which yields a theoretical convergence rate at least as good as IProx-SDCA~\cite{IProx-SDCA}. However,  there are two main implementation issues of AdaSDCA:
 \begin{itemize}
 \item[1.] The update of the dual residue $\kappa^t$ at each iteration costs $O(\mynnz(A))$ where $\mynnz(A)$ is the number of nonzero elements of the matrix $A$; 
 \item[2.] We do not know how to compute the optimal solution of~\eqref{eq:optp}.
 \end{itemize}

  In this section, we propose a heuristic variant of AdaSDCA, which avoids the above two issues while staying close to the 'good' adaptive sampling distribution.
 
 \subsection{Description of Algorithm}

\begin{algorithm}
\caption{AdaSDCA+} \label{alg:HA}
\begin{algorithmic}
\STATE \textbf{Parameter} a number $m>1$
\STATE \textbf{Initialization} Choose $\alpha^0 \in \R^n$, set $\bar \alpha^ 0=\frac{1}{\lambda n} A \alpha^ 0$
\FOR{$t \geq 0$}
\STATE Primal update: $w^t =  \nabla g^* \left(\bar\alpha^{t}\right)$
\STATE Set: $\alpha^{t+1}=\alpha^{t}$
\IF{$\mod(t, n) == 0$}
\STATE \textbf{Option I:} Adaptive probability
\STATE \quad Compute: $\kappa_i^t=\alpha_i^t+\nabla \phi_i(A_i^ \top w^ t) , ~ \forall i \in [n]$
\STATE \quad Set: $p_i^t \sim |\kappa_i^t|\sqrt{v_{i} + n\lambda\gamma},~ \forall i \in [n]$ 
\STATE \textbf{Option II:} Optimal Importance probability
\STATE \quad Set: $p_i^t \sim (v_i + n\lambda\gamma),\enspace \forall i\in [n]$
\ENDIF
\STATE Generate random $i_t \in [n]$ according to $p^t$
\STATE Compute: \\
 $\qquad \Delta\alpha_{i_t}^t = \displaystyle\argmax_{\Delta \in \R} \left\{-\phi_{i_t}^*(-(\alpha_{i_t}^{t}+\Delta))\right.$\\
$\qquad \qquad\qquad\qquad\quad\left.-A_{i_t}^\top w^t\Delta-
\frac{v_{i_t}}{2\lambda n}|\Delta|^2\right\}$
\STATE Dual update: $\alpha_{i_t}^{t+1} = \alpha_{i_t}^{t} + \Delta\alpha_{i_t}^t$
\STATE Average update: $\bar \alpha^t=\bar \alpha^t+\frac{\Delta\alpha_{i_t}}{\lambda n}A_{i_t}$
\STATE Probability update:
\\ $\qquad \qquad p_{i_t}^{t+1}\sim p_{i_t}^t/m$, $\enspace p_{j}^{t+1}\sim p_{j}^t$, $\forall j \neq i_t$ \\
\ENDFOR
\STATE \textbf{Output:} $w^t, \alpha^t$
\end{algorithmic}
\end{algorithm}

AdaSDCA+ has the same structure as AdaSDCA with a few important differences.

\textbf{Epochs}  AdaSDCA+ is divided into epochs of length $n$. At the beginning of every epoch, sampling probabilities are computed according to one of two options. During each epoch the probabilities are cheaply updated at the end of every iteration to approximate the adaptive model. The intuition behind is as follows. After  $i$ is sampled and the dual coordinate $\alpha_i$ is updated, the residue $\kappa_i$ naturally decreases. We then decrease also the probability that $i$ is chosen in the next iteration, by setting $p^{t+1}$ to be proportional to 
$(p_1^t,\dots p_{i-1}^t, p_i^t/m, p_{i+1}^t,\dots, p_n^t)$. By doing this we avoid the computation of $\kappa$ at each iteration (issue 1) which costs as much as the full gradient algorithm, while following closely the changes of the dual residue $\kappa$. We  reset the adaptive sampling probability after every epoch of length $n$. 


\textbf{Parameter $m$} The setting of parameter $m$ in AdaSDCA+ directly affects the performance of the algorithm. If  $m$ is too large, the probability of sampling the same coordinate twice during an epoch will be very small. This will result in a random permutation through all coordinates every epoch. On the other hand, for $m$ too small the coordinates having larger probabilities at the beginning of an epoch could be sampled more often than it should, even after their corresponding dual residues become sufficiently small. We don't have a definitive rule on the choice of $m$ and we leave this to future work. Experiments with different choices of $m$ can be found in Section \ref{s:Experiments}.

\textbf{Option I \& Option II} At the beginning of each epoch, one can choose between two options for resetting the sampling probability. Option I corresponds to the optimal solution of~\eqref{eq:optprelaxed}, given by the closed form~\eqref{a:optipsdf}. Option II is the optimal serial sampling probability~\eqref{a:pstar}, the same as the one used in IProx-SDCA~\cite{IProx-SDCA}. However, AdaSDCA+ differs significantly with IProx-SDCA since we also update iteratively the sampling probability, which as we show through numerical experiments yields a faster convergence than IProx-SDCA. 


\subsection{Computational cost}

\textbf{Sampling and probability update} During the algorithm we sample $i\in [n]$ from non-uniform probability distribution $p^t$, which changes at each iteration. This process can be done efficiently using the {Random Counters} algorithm introduced in Section 6.2 of \cite{Nesterov:2010RCDM}, which takes $O(n\log(n))$ operations to create the probability tree and $O(\log(n))$ operations to sample from the distribution or change one of the probabilities.  

\textbf{Total computational cost} We can compute the computational cost of one epoch. At the beginning of an epoch, we need $O(\mynnz)$ operations to calculate the dual residue $\kappa$. Then we create a probability tree using $O(n\log(n))$ operations. At each iteration we need $O(\log(n))$ operations to sample a coordinate, $O(\mynnz/n)$ operations to calculate the update to $\alpha$ and a further $O( \log(n))$ operations to update the probability tree. As a result an epoch needs $O(\mynnz + n\log(n))$ operations. 
For comparison purpose we list in Table~\ref{t:comp_costs} the one epoch computational cost of comparable algorithms.

\begin{table}[t]
\caption{One epoch computational cost of different algorithms}
\label{t:comp_costs}
\vskip 0.15in
\begin{center}
\begin{small}
\begin{sc}
\begin{tabular}{|l|c|}
\hline
\abovespace\belowspace
Algorithm & cost of an epoch  \\
\hline
\abovespace
SDCA\& QUARTZ(uniform)  & $O(\mynnz)$ \\
\hline
\abovespace
IProx-SDCA & $O(\mynnz + n \log(n))$ \\
\hline
\abovespace
AdaSDCA & $O(n\cdot \mynnz) $ \\
\hline
\abovespace
AdaSDCA+ & $O(\mynnz + n\log(n))$ \\
\hline
\end{tabular}
\end{sc}
\end{small}
\end{center}
\vskip -0.1in
\end{table}

\section{Numerical Experiments} \label{s:Experiments}

In this section we present results of numerical experiments.

\subsection{Loss functions}

We test AdaSDCA and AdaSDCA+, SDCA, and IProx-SDCA for two different types of loss functions $\{\phi_i\}_{i=1}^n$: quadratic loss and smoothed Hinge loss. Let $y\in \R^n$ be the vector of labels. The  quadratic loss is given by
$$\phi_i(x) = \frac{1}{2\gamma}(x - y_i)^2$$
 and the smoothed Hinge loss is:
 \begin{equation*}
\phi_i(x) = \begin{cases}
0 &y_i x \geq 1
\\ 1- y_i x - \gamma/2 & y_i x \leq 1-\gamma
\\ \frac{(1 - y_i x)^2}{2\gamma} &\text{otherwise,}
\end{cases}
\end{equation*}
In both cases we use $L_2$-regularizer, i.e.,
$$
g(w)=\frac{1}{2}\|w\|^2.
$$

 Quadratic loss  functions appear usually in regression problems, and smoothed Hinge loss can be found in linear support vector machine (SVM) problems~\cite{ASDCA}.

\subsection{Numerical results}

 We used 5 different datasets: w8a, dorothea, mushrooms, cov1 and ijcnn1 (see Table \ref{t:datasets}). 
 
\begin{table}[t]
\caption{Dimensions and nonzeros of the datasets}
\label{t:datasets}
\vskip 0.15in
\begin{center}
\begin{small}
\begin{sc}
\begin{tabular}{|l|c|c|c|}
\hline
\abovespace\belowspace
Dataset & $d$ & $n$ & $\mynnz/(nd)$  \\
\hline
\abovespace
w8a  & $300$ & $49,749$ & $3.9\%$ \\
\hline
\abovespace
dorothea & $100,000$ & 800 & $0.9\%$  \\
\hline
\abovespace
mushrooms & $112$ & $8,124$ & $18.8\%$ \\
\hline
\abovespace
cov1 & $54$ & $581,012$ & $22\%$ \\
\hline
\abovespace
ijcnn1 & $22$ & $49,990$ & $41\%$ \\
\hline
\end{tabular}
\end{sc}
\end{small}
\end{center}
\vskip -0.1in
\end{table}
 
 In all our experiments we used $\gamma=1$ and $\lambda=1/n$. 


\textbf{AdaSDCA} The results of the theory developed in Section~\ref{s:conv} can be observed through Figure~\ref{w8a_fullad} to Figure~\ref{ijcnn1_fullad}. AdaSDCA needs the least amount of iterations to converge, confirming the theoretical result.

\textbf{AdaSDCA+ V.S.  others} We can observe through Figure~\ref{w8a_quad_it} to~\ref{ijcnn1_hinge_it}, that both options of AdaSDCA+ outperforms SDCA and IProx-SDCA, in terms of number of iterations, for quadratic loss functions and for smoothed Hinge loss functions. One can observe similar results in terms of time through Figure~\ref{time_w8a_quad} to Figure~\ref{time_ijcnn1_hinge}.

\textbf{Option I V.S. Option II}
Despite the fact that Option I is not theoretically supported for smoothed hinge loss, it still converges faster than Option II on every dataset and for every loss function. The biggest difference can be observed on Figure \ref{time_cov1_hinge}, where Option I converges to the machine precision in just 15 seconds.

\textbf{Different choices of $m$} 
 To show the impact of different choices of $m$ on the performance of AdaSDCA+, in Figures~\ref{w8a_quad_itm} to~\ref{ijcnn1_hinge_itm} we compare the results of the two options of AdaSDCA+ using different $m$ equal to $2$, $10$ and $50$.
It is hard to draw a clear conclusion here because clearly the optimal $m$ shall depend on the dataset and the problem type. 


\begin{figure}[ht]
\vskip 0.2in
\begin{center}
\centerline{\includegraphics[width=0.99\columnwidth]{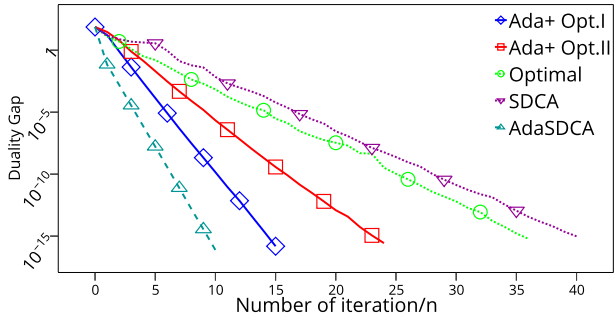}}
\caption{w8a dataset $d$ = 300, $n$ = 49749, Quadratic loss with $L_2$ regularizer, comparing number of iterations with known algorithms}
\label{w8a_fullad}
\end{center}
\vskip -0.2in
\end{figure}

\begin{figure}[t]
\vskip 0.2in
\begin{center}
\centerline{\includegraphics[width=0.99\columnwidth]{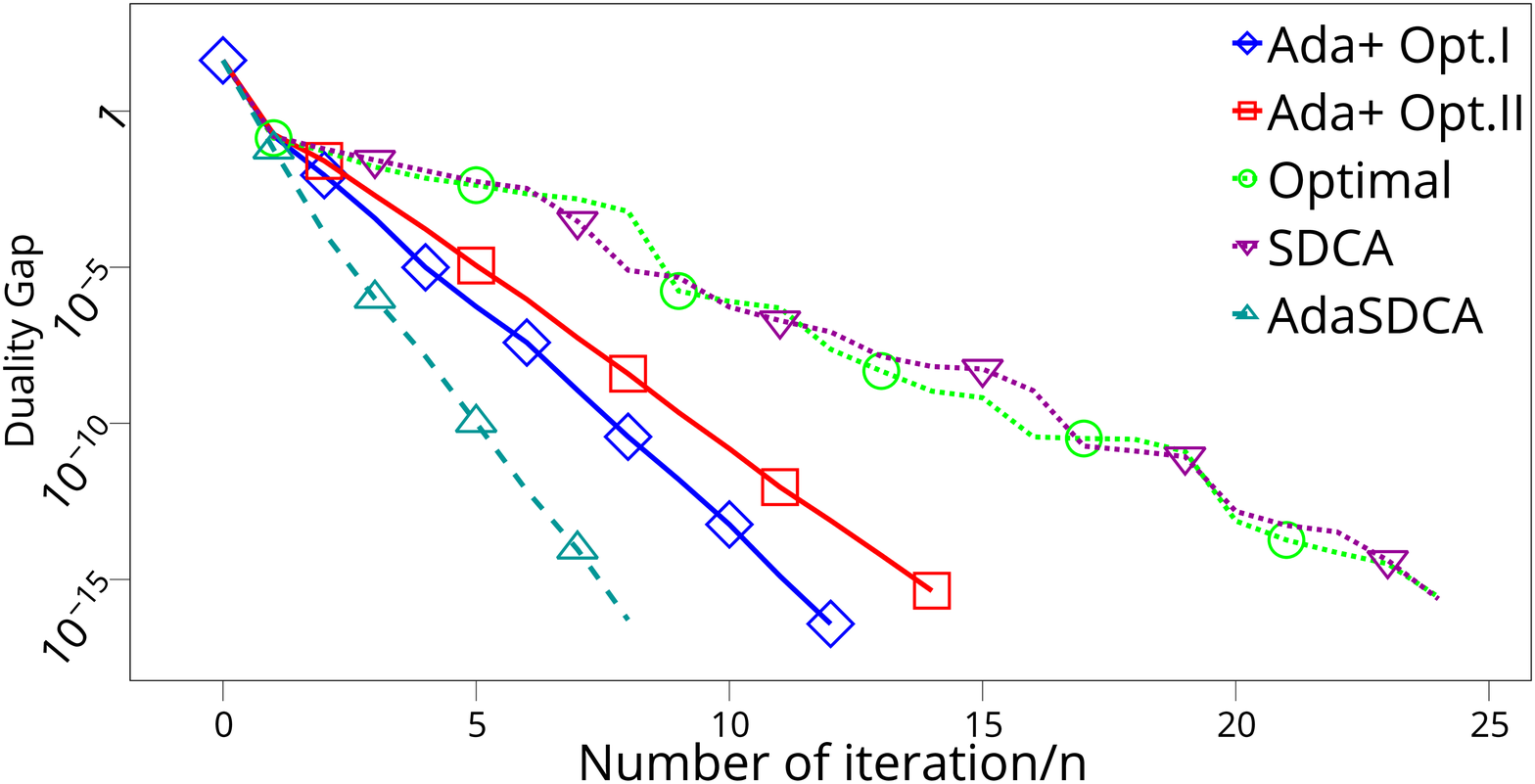}}
\caption{dorothea dataset $d$ = 100000, $n$ = 800, Quadratic loss with $L_2$ regularizer, comparing number of iterations with known algorithms}
\label{dorothea_fullad}
\end{center}
\vskip -0.2in
\end{figure}

\begin{figure}[t]
\vskip 0.2in
\begin{center}
\centerline{\includegraphics[width=0.99\columnwidth]{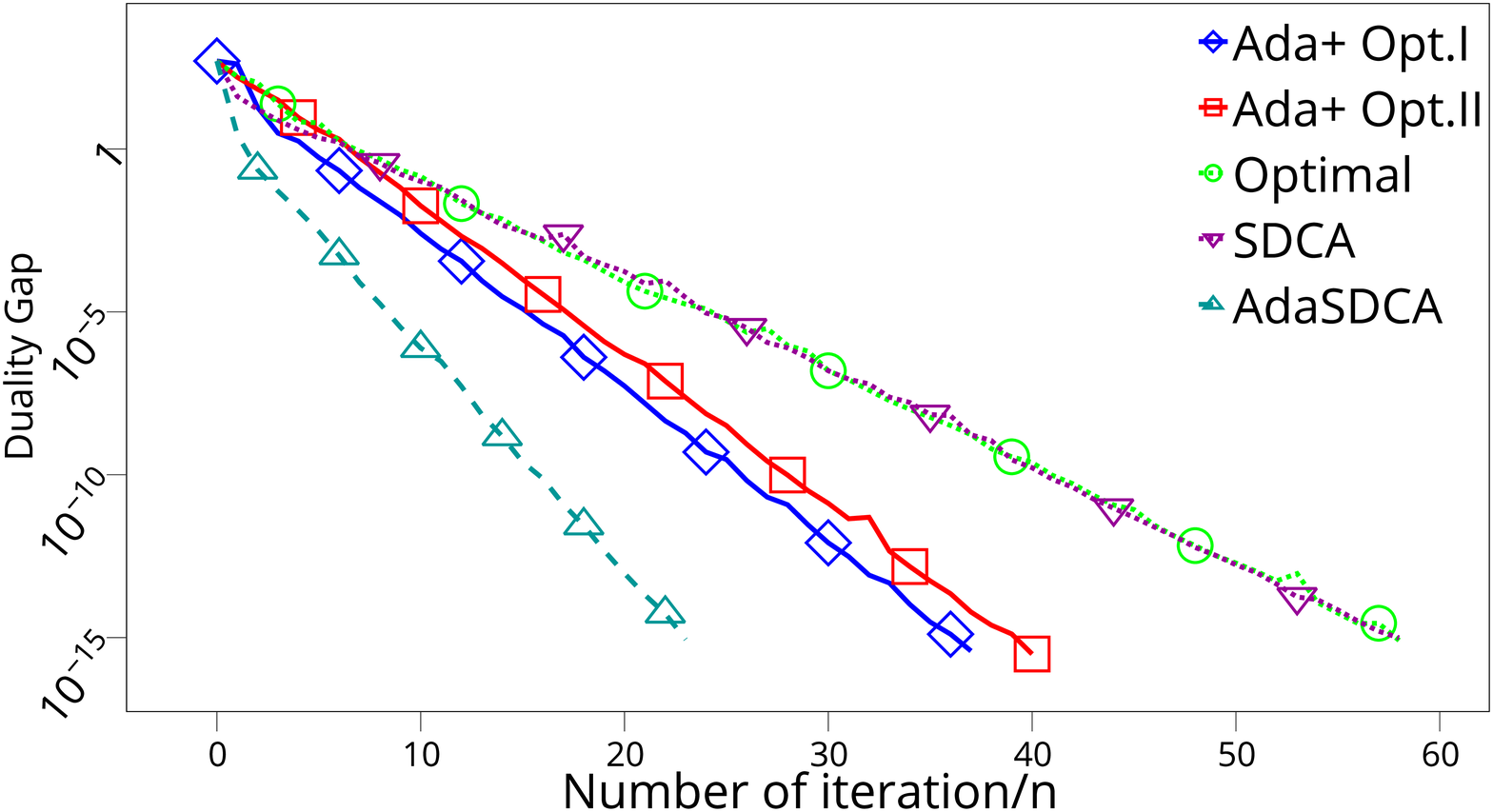}}
\caption{mushrooms dataset $d$ = 112, $n$ = 8124, Quadratic loss with $L_2$ regularizer, comparing number of iterations with known algorithms}
\label{mushrooms_fullad}
\end{center}
\vskip -0.2in
\end{figure}

\begin{figure}[t]
\vskip 0.2in
\begin{center}
\centerline{\includegraphics[width=0.99\columnwidth]{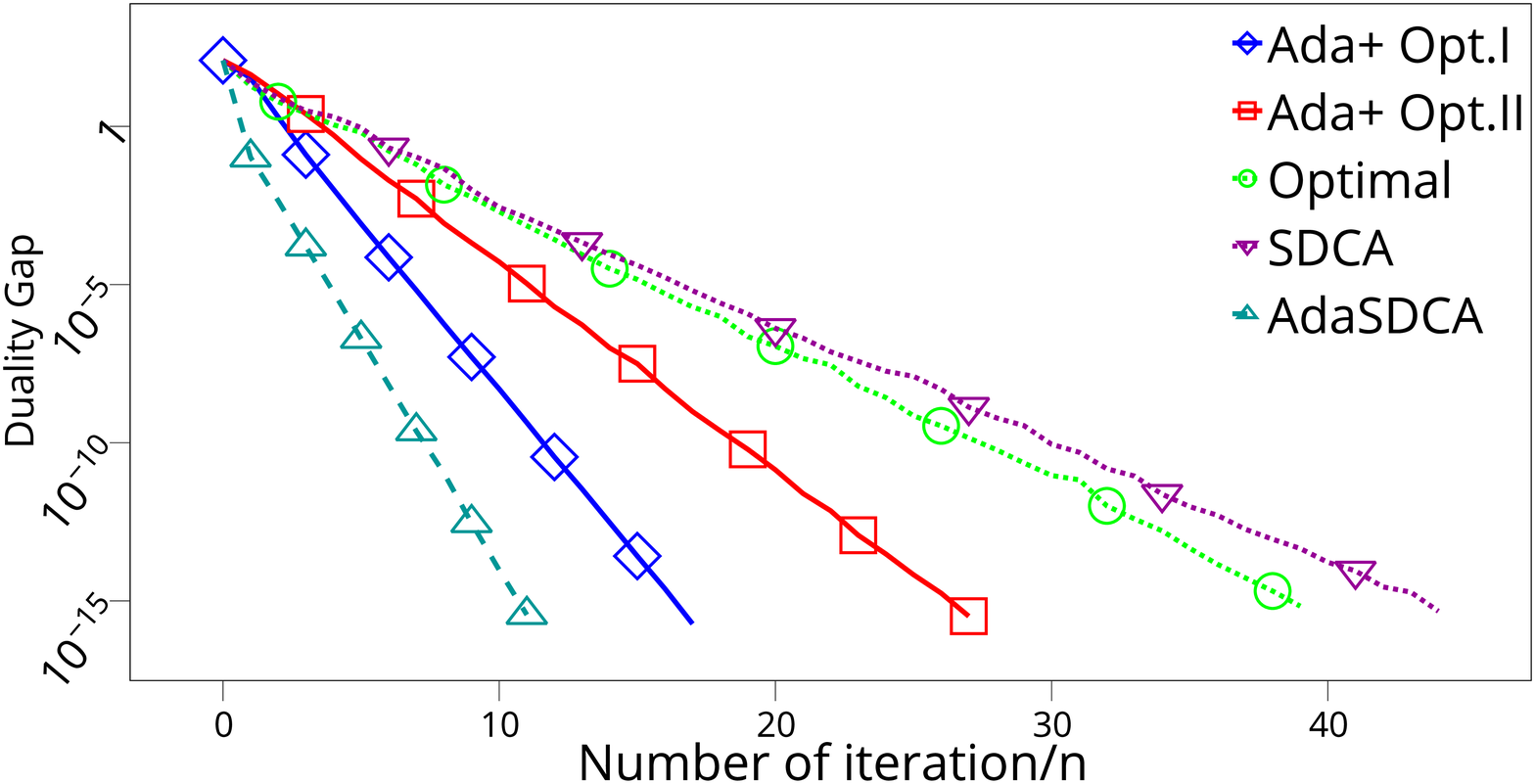}}
\caption{ijcnn1 dataset $d$ = 22, $n$ = 49990, Quadratic loss with $L_2$ regularizer, comparing number of iterations with known algorithms}
\label{ijcnn1_fullad}
\end{center}
\vskip -0.2in
\end{figure}

\clearpage

\begin{figure}[t]
\vskip 0.2in
\begin{center}
\centerline{\includegraphics[width=0.99\columnwidth]{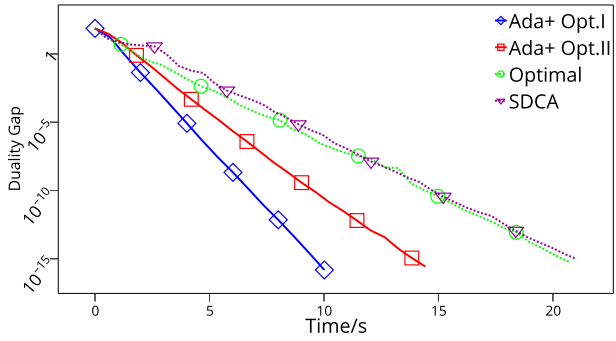}}
\caption{w8a dataset $d$ = 300, $n$ = 49749, Quadratic loss with $L_2$ regularizer, comparing real time with known algorithms}
\label{time_w8a_quad}
\end{center}
\vskip -0.2in
\end{figure}

\begin{figure}[t]
\vskip 0.2in
\begin{center}
\centerline{\includegraphics[width=0.99\columnwidth]{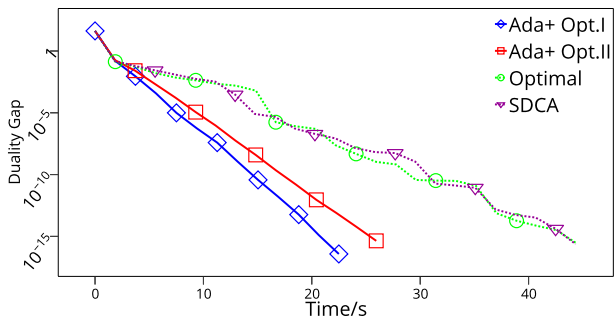}}
\caption{dorothea dataset $d$ = 100000, $n$ = 800, Quadratic loss with $L_2$ regularizer, comparing real time with known algorithms}
\label{time_dorothea_quad}
\end{center}
\vskip -0.2in
\end{figure}

\begin{figure}[t]
\vskip 0.2in
\begin{center}
\centerline{\includegraphics[width=0.99\columnwidth]{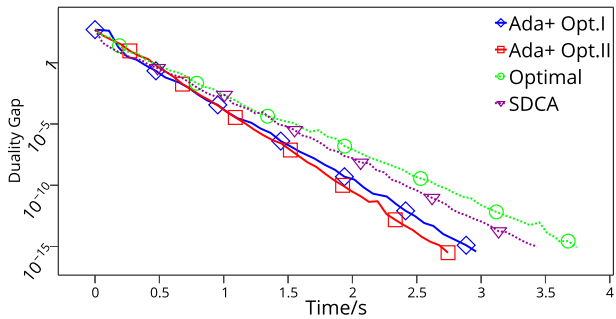}}
\caption{mushrooms dataset $d$ = 112, $n$ = 8124, Quadratic loss with $L_2$ regularizer, comparing real time with known algorithms}
\label{time_mushrooms_quad}
\end{center}
\vskip -0.2in
\end{figure}

\begin{figure}[t]
\vskip 0.2in
\begin{center}
\centerline{\includegraphics[width=0.99\columnwidth]{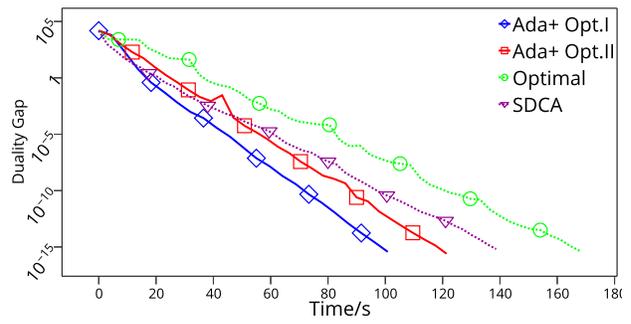}}
\caption{cov1 dataset $d$ = 54, $n$ = 581012, Quadratic loss with $L_2$ regularizer, comparing real time with known algorithms}
\label{time_cov1_quad}
\end{center}
\vskip -0.2in
\end{figure}

\begin{figure}[t]
\vskip 0.2in
\begin{center}
\centerline{\includegraphics[width=0.99\columnwidth]{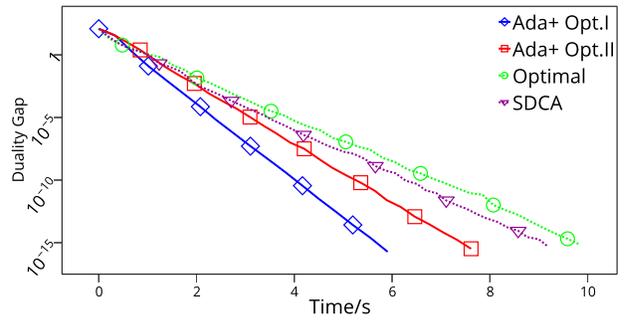}}
\caption{ijcnn1 dataset $d$ = 22, $n$ = 49990, Quadratic loss with $L_2$ regularizer, comparing real time with known algorithms}
\label{time_ijcnn1_quad}
\end{center}
\vskip -0.2in
\end{figure}

\begin{figure}[t]
\vskip 0.2in
\begin{center}
\centerline{\includegraphics[width=0.99\columnwidth]{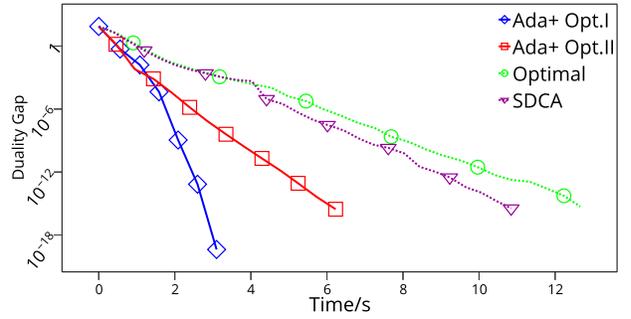}}
\caption{w8a dataset $d$ = 300, $n$ = 49749, Smooth Hinge loss with $L_2$ regularizer, comparing real time with known algorithms}
\label{time_w8a_hinge}
\end{center}
\vskip -0.2in
\end{figure}

\clearpage

\begin{figure}[t]
\vskip 0.2in
\begin{center}
\centerline{\includegraphics[width=0.99\columnwidth]{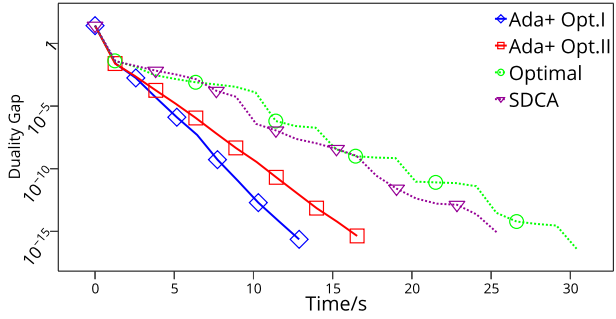}}
\caption{dorothea dataset $d$ = 100000, $n$ = 800, Smooth Hinge loss with $L_2$ regularizer, comparing real time with known algorithms}
\label{time_dorothea_hinge}
\end{center}
\vskip -0.2in
\end{figure}

\begin{figure}[t]
\vskip 0.2in
\begin{center}
\centerline{\includegraphics[width=0.99\columnwidth]{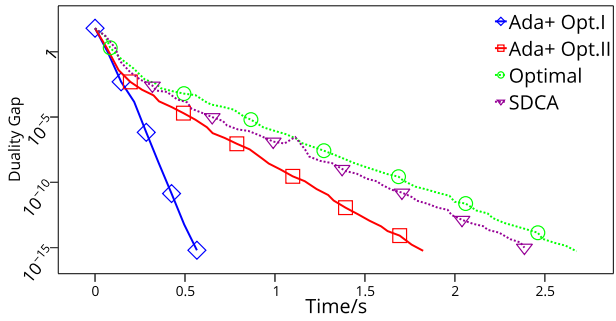}}
\caption{mushrooms dataset $d$ = 112, $n$ = 8124, Smooth Hinge loss with $L_2$ regularizer, comparing real time with known algorithms}
\label{time_mushrooms_hinge}
\end{center}
\vskip -0.2in
\end{figure}

\begin{figure}[t]
\vskip 0.2in
\begin{center}
\centerline{\includegraphics[width=0.99\columnwidth]{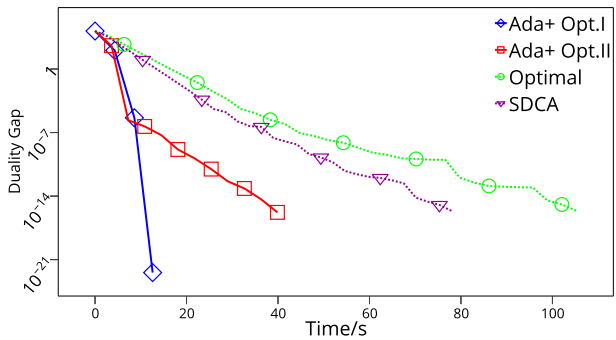}}
\caption{cov1 dataset $d$ = 54, $n$ = 581012, Smooth Hinge loss with $L_2$ regularizer, comparing real time with known algorithms}
\label{time_cov1_hinge}
\end{center}
\vskip -0.2in
\end{figure}

\begin{figure}[t]
\vskip 0.2in
\begin{center}
\centerline{\includegraphics[width=0.99\columnwidth]{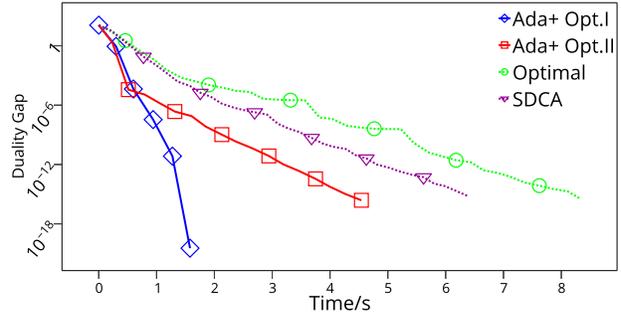}}
\caption{ijcnn1 dataset $d$ = 22, $n$ = 49990, Smooth Hinge loss with $L_2$ regularizer, comparing real time with known algorithms}
\label{time_ijcnn1_hinge}
\end{center}
\vskip -0.2in
\end{figure}  

\begin{figure}[t]
\vskip 0.2in
\begin{center}
\centerline{\includegraphics[width=0.99\columnwidth]{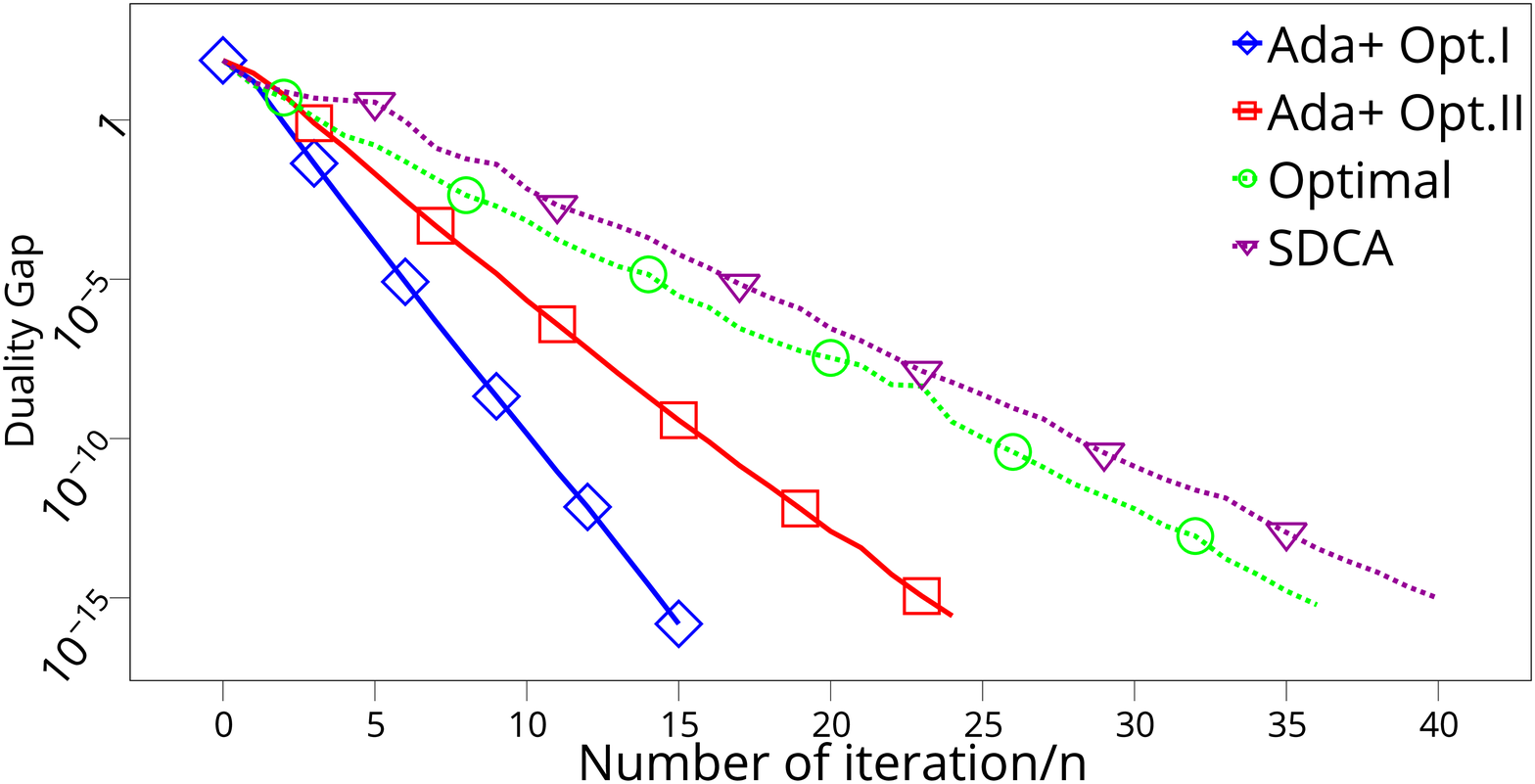}}
\caption{w8a dataset $d$ = 300, $n$ = 49749, Quadratic loss with $L_2$ regularizer, comparing number of iterations with known algorithms}
\label{w8a_quad_it}
\end{center}
\vskip -0.2in
\end{figure}

\begin{figure}[t]
\vskip 0.2in
\begin{center}
\centerline{\includegraphics[width=0.99\columnwidth]{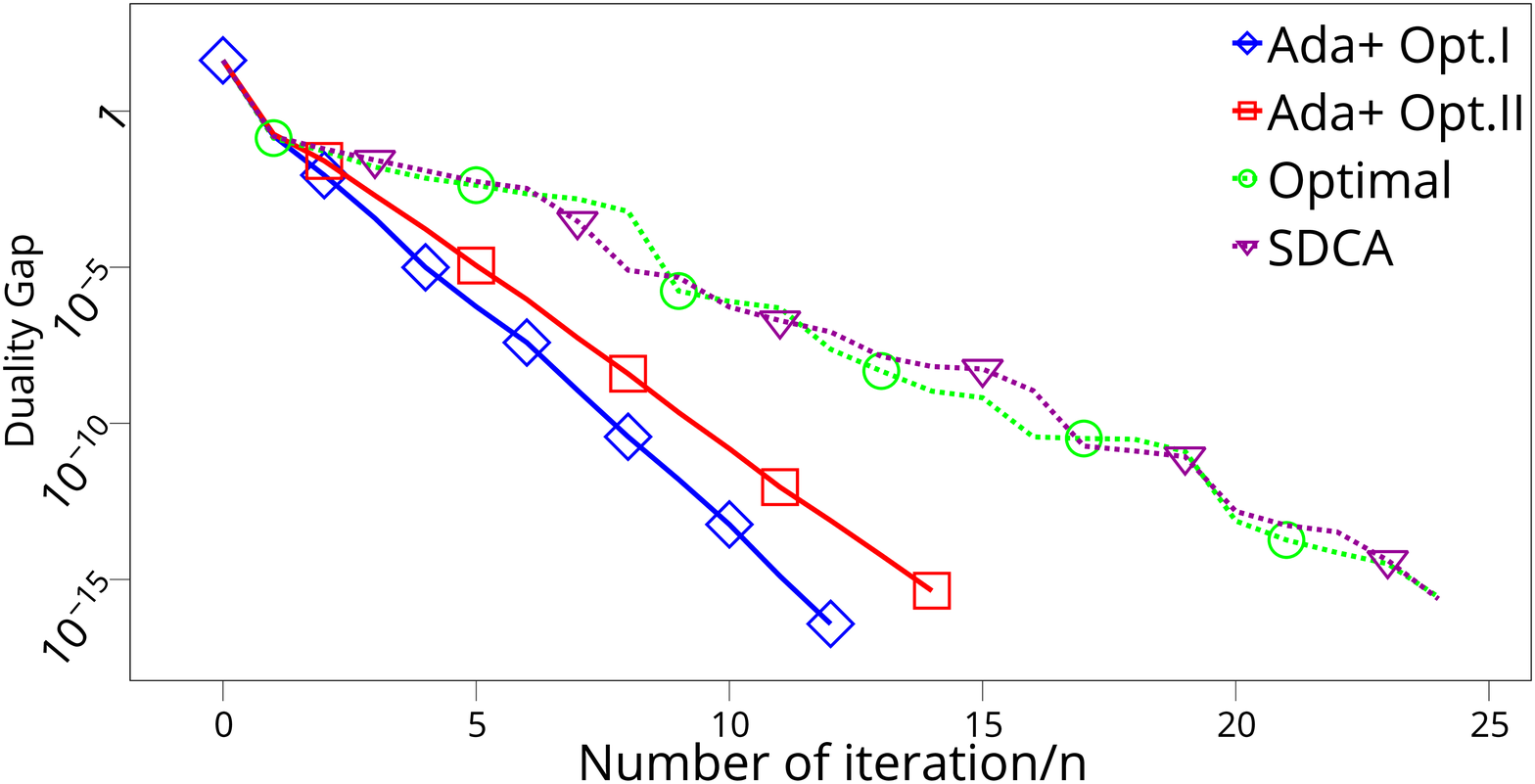}}
\caption{dorothea dataset $d$ = 100000, $n$ = 800, Quadratic loss with $L_2$ regularizer, comparing number of iterations with known algorithms}
\label{dorothea_quad_it}
\end{center}
\vskip -0.2in
\end{figure}

\clearpage

\begin{figure}[t]
\vskip 0.2in
\begin{center}
\centerline{\includegraphics[width=0.99\columnwidth]{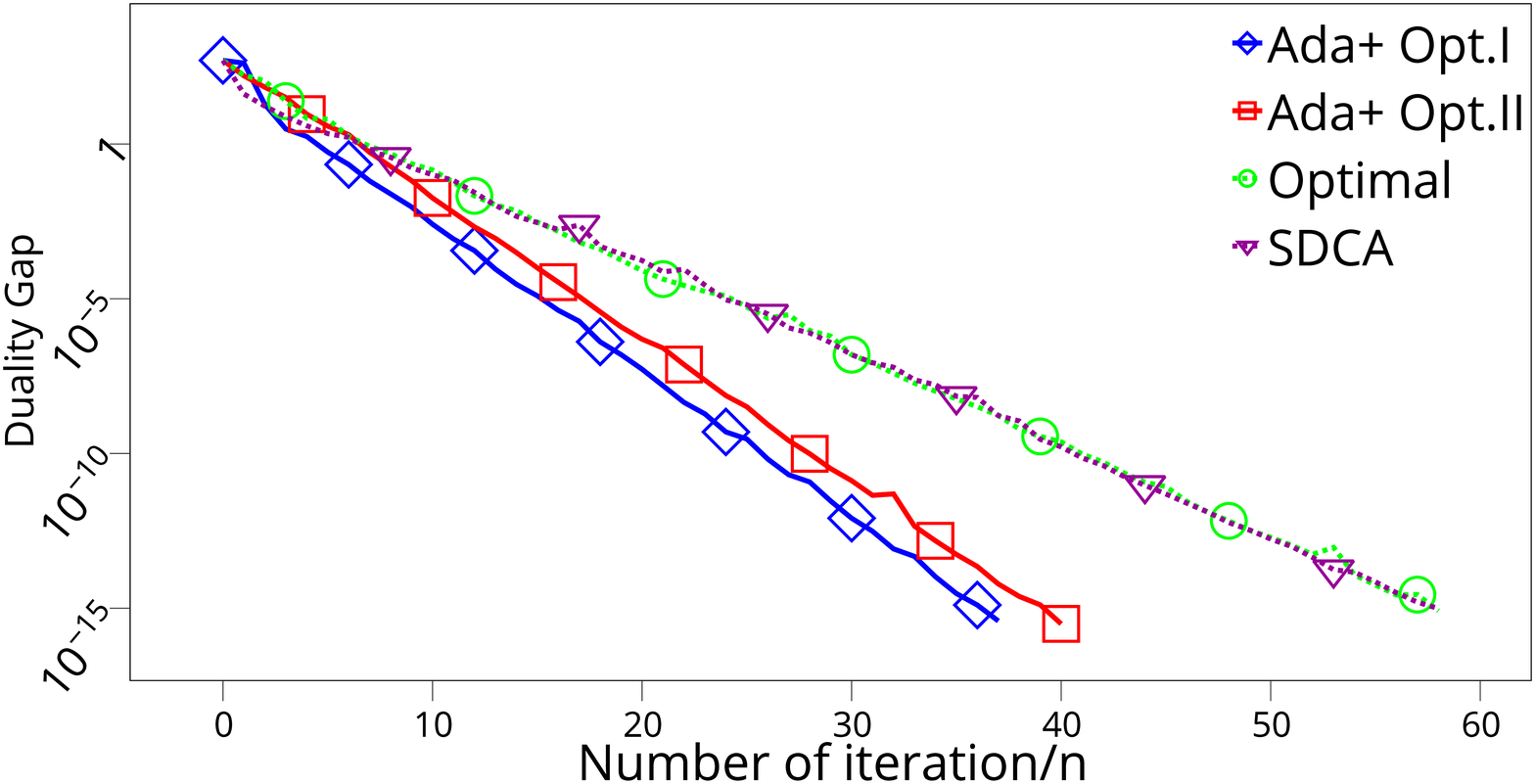}}
\caption{mushrooms dataset $d$ = 112, $n$ = 8124, Quadratic loss with $L_2$ regularizer, comparing number of iterations with known algorithms}
\label{mushrooms_quad_it}
\end{center}
\vskip -0.2in
\end{figure}

\begin{figure}[t]
\vskip 0.2in
\begin{center}
\centerline{\includegraphics[width=0.99\columnwidth]{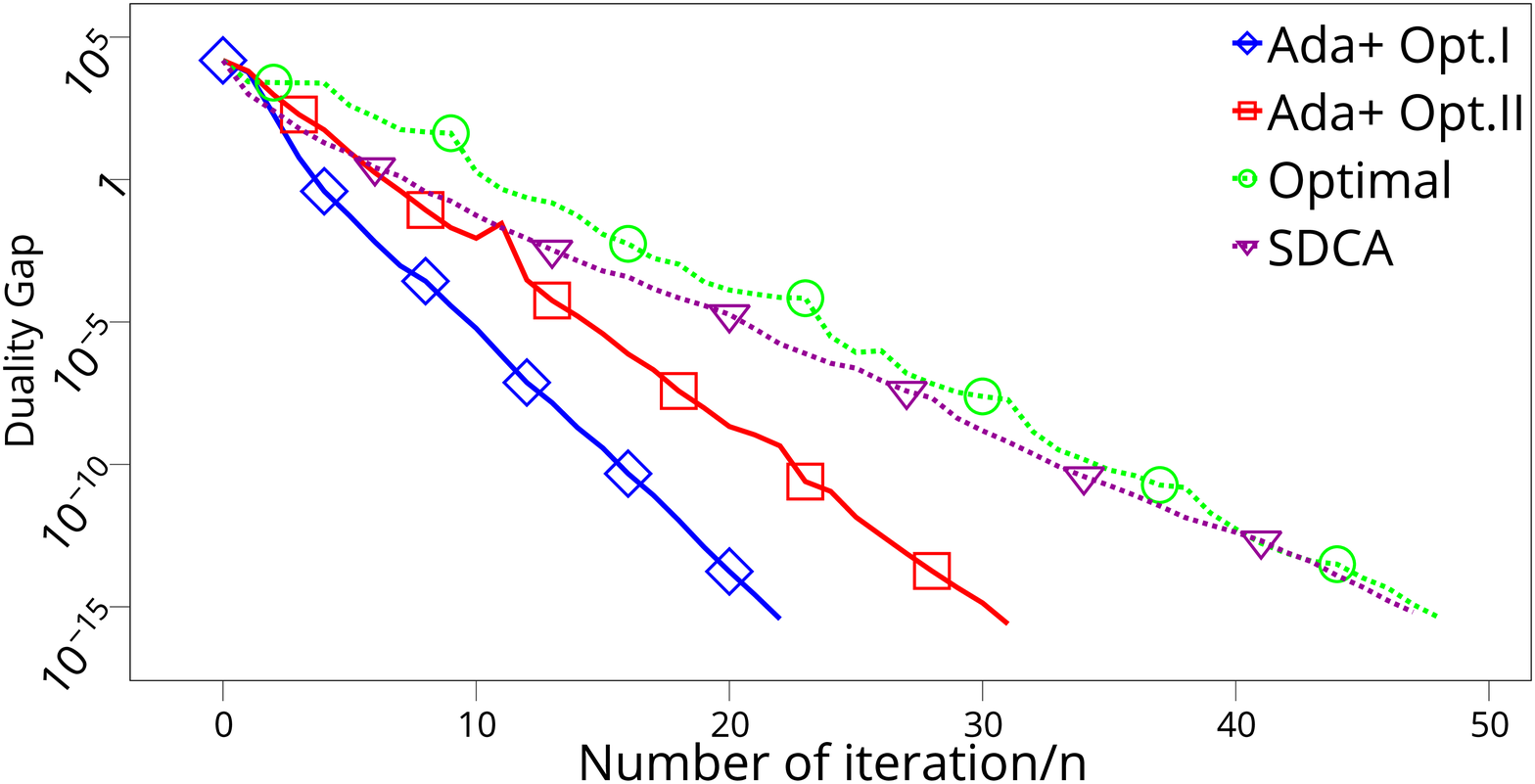}}
\caption{cov1 dataset $d$ = 54, $n$ = 581012, Quadratic loss with $L_2$ regularizer, comparing number of iterations with known algorithms}
\label{cov1_quad_it}
\end{center}
\vskip -0.2in
\end{figure}

\begin{figure}[t]
\vskip 0.2in
\begin{center}
\centerline{\includegraphics[width=0.99\columnwidth]{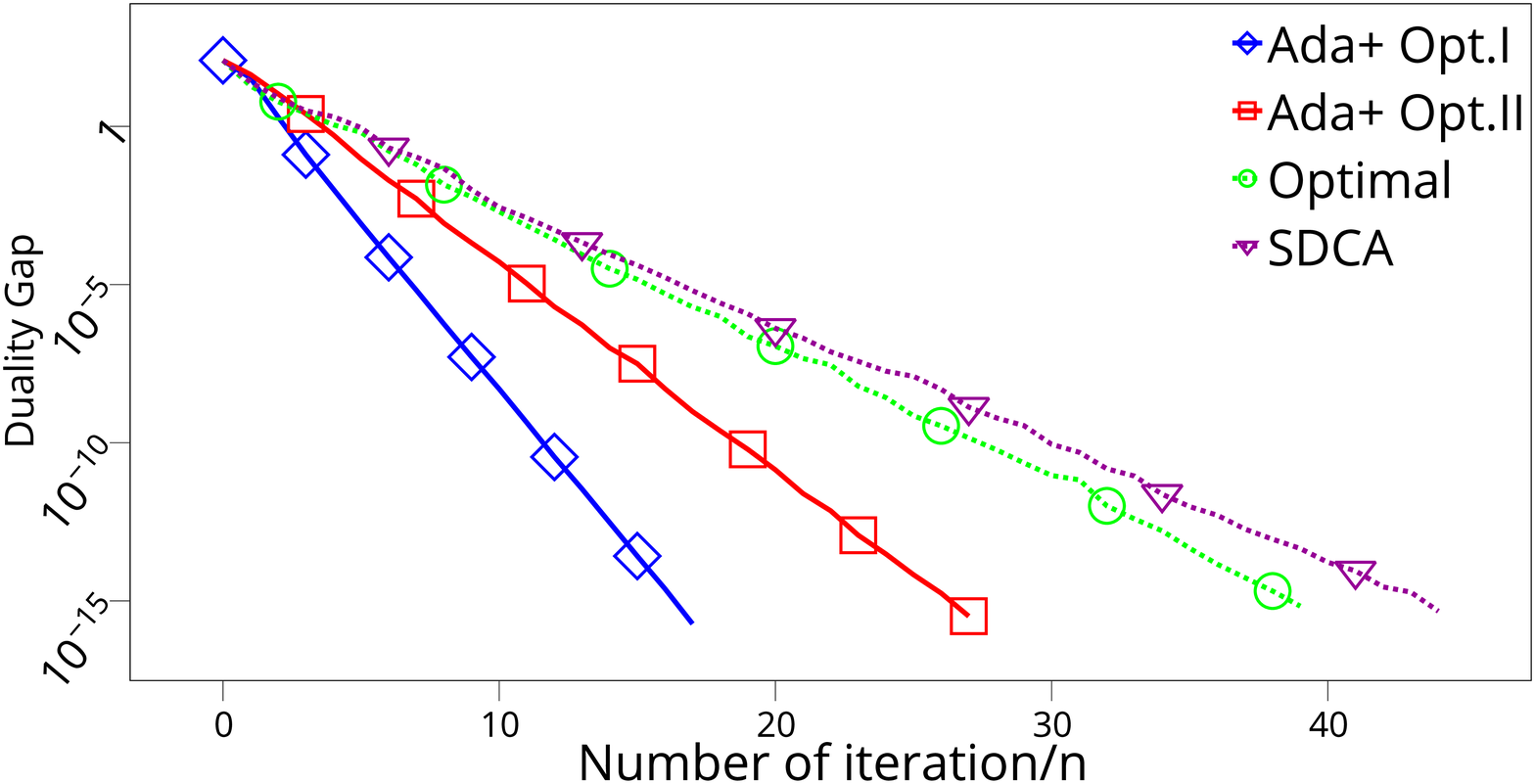}}
\caption{ijcnn1 dataset $d$ = 22, $n$ = 49990, Quadratic loss with $L_2$ regularizer, comparing number of iterations with known algorithms}
\label{ijcnn1_quad_it}
\end{center}
\vskip -0.2in
\end{figure}

\begin{figure}[t]
\vskip 0.2in
\begin{center}
\centerline{\includegraphics[width=0.99\columnwidth]{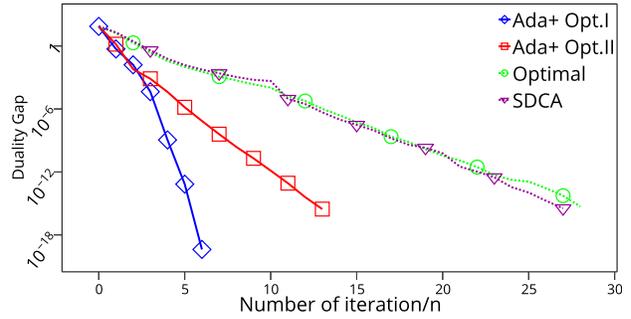}}
\caption{w8a dataset $d$ = 300, $n$ = 49749, Smooth Hinge loss with $L_2$ regularizer, comparing number of iterations with known algorithms}
\label{w8a_hinge_it}
\end{center}
\vskip -0.2in
\end{figure}

\begin{figure}[t]
\vskip 0.2in
\begin{center}
\centerline{\includegraphics[width=0.99\columnwidth]{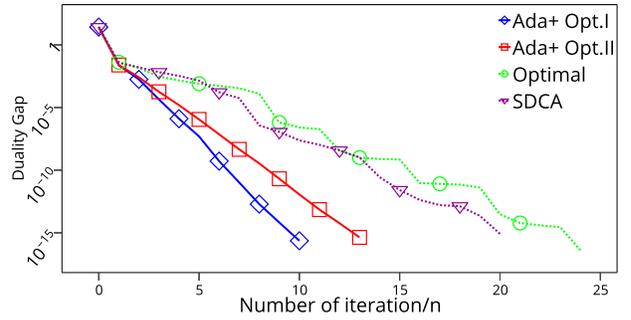}}
\caption{dorothea dataset $d$ = 100000, $n$ = 800, Smooth Hinge loss with $L_2$ regularizer, comparing number of iterations with known algorithms}
\label{dorothea_hinge_it}
\end{center}
\vskip -0.2in
\end{figure}

\begin{figure}[t]
\vskip 0.2in
\begin{center}
\centerline{\includegraphics[width=0.99\columnwidth]{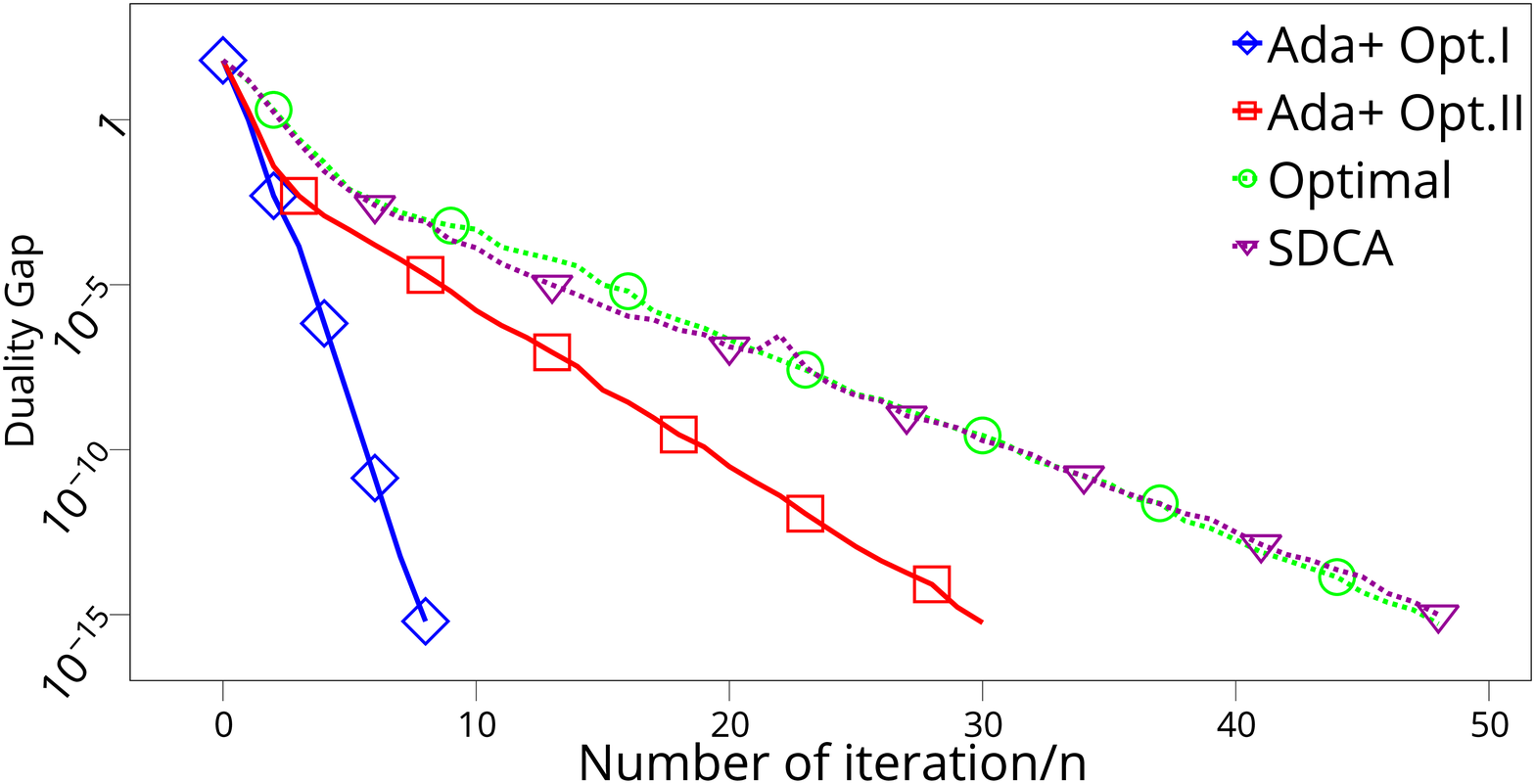}}
\caption{mushrooms dataset $d$ = 112, $n$ = 8124, Smooth Hinge loss with $L_2$ regularizer, comparing number of iterations with known algorithms}
\label{mushrooms_hinge_it}
\end{center}
\vskip -0.2in
\end{figure}

\clearpage

\begin{figure}[t]
\vskip 0.2in
\begin{center}
\centerline{\includegraphics[width=0.99\columnwidth]{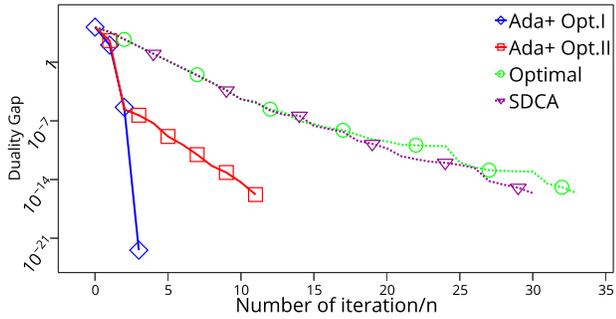}}
\caption{cov1 dataset $d$ = 54, $n$ = 581012, Smooth Hinge loss with $L_2$ regularizer, comparing number of iterations with known algorithms}
\label{cov1_hinge_it}
\end{center}
\vskip -0.2in
\end{figure}

\begin{figure}[t]
\vskip 0.2in
\begin{center}
\centerline{\includegraphics[width=0.99\columnwidth]{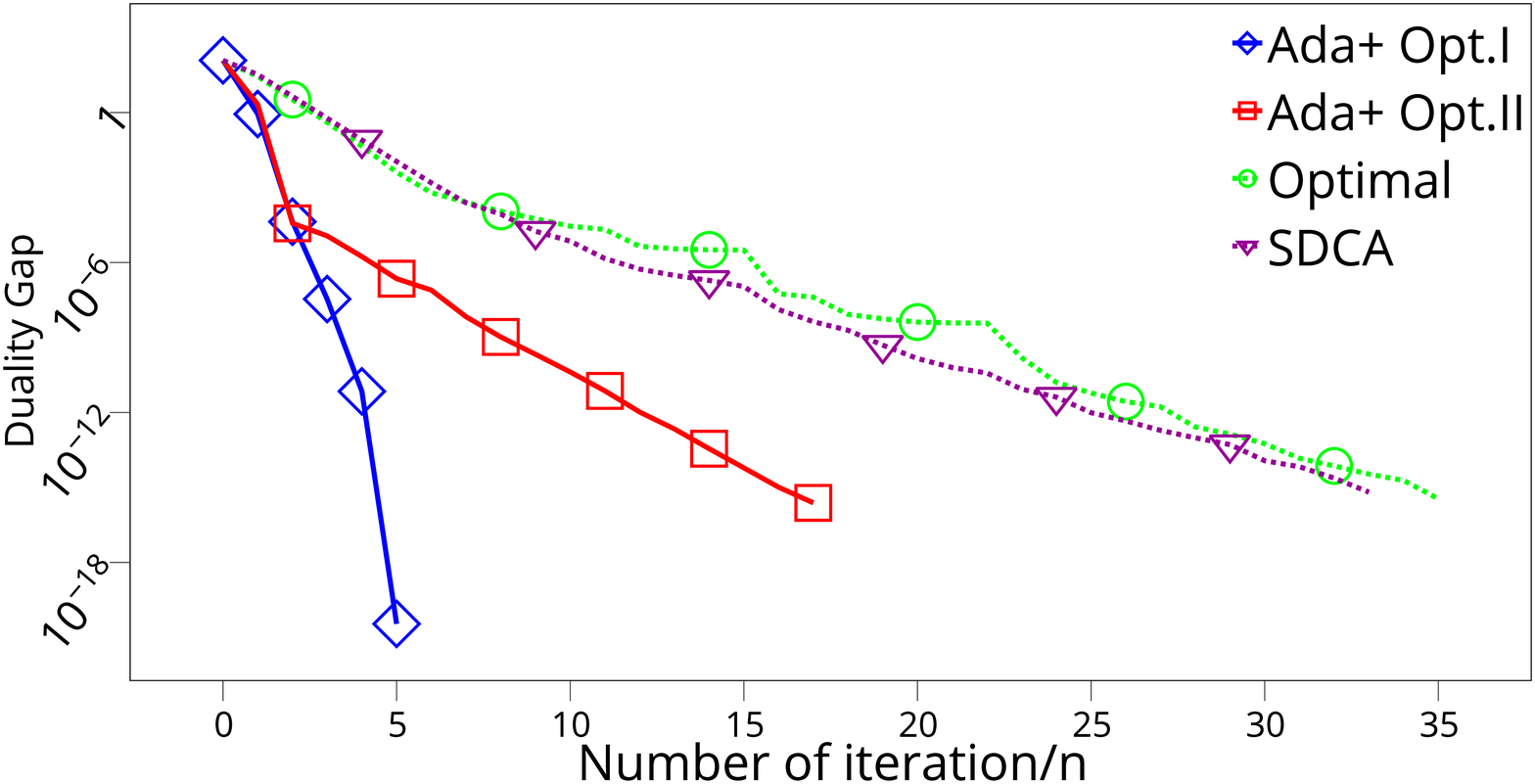}}
\caption{ijcnn1 dataset $d$ = 22, $n$ = 49990, Smooth Hinge loss with $L_2$ regularizer, comparing number of iterations with known algorithms}
\label{ijcnn1_hinge_it}
\end{center}
\vskip -0.2in
\end{figure}

\begin{figure}[t]
\vskip 0.2in
\begin{center}
\centerline{\includegraphics[width=0.99\columnwidth]{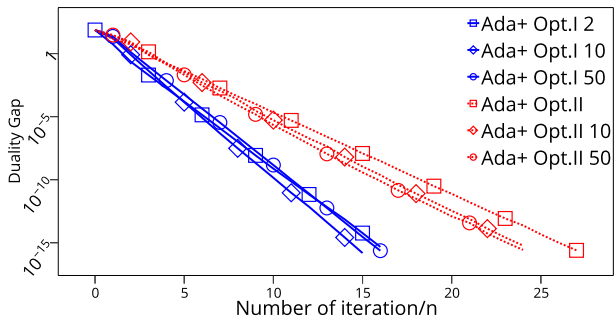}}
\caption{w8a dataset $d$ = 300, $n$ = 49749, Quadratic loss with $L_2$ regularizer, comparison of different choices of the constant $m$}
\label{w8a_quad_itm}
\end{center}
\vskip -0.2in
\end{figure}

\begin{figure}[t]
\vskip 0.2in
\begin{center}
\centerline{\includegraphics[width=0.99\columnwidth]{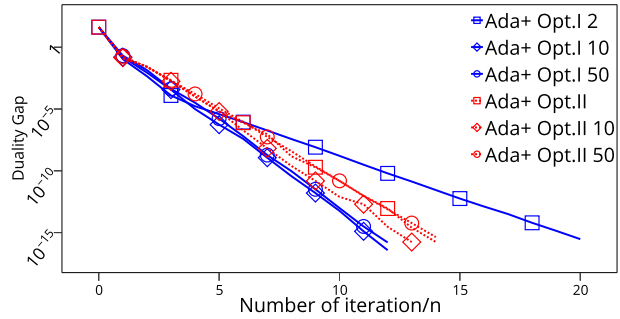}}
\caption{dorothea dataset $d$ = 100000, $n$ = 800, Quadratic loss with $L_2$ regularizer, comparison of different choices of the constant $m$}
\label{dorothea_quad_itm}
\end{center}
\vskip -0.2in
\end{figure}

\begin{figure}[t]
\vskip 0.2in
\begin{center}
\centerline{\includegraphics[width=0.99\columnwidth]{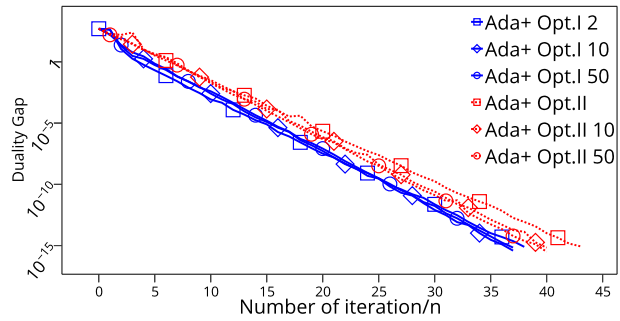}}
\caption{mushrooms dataset $d$ = 112, $n$ = 8124, Quadratic loss with $L_2$ regularizer, comparison of different choices of the constant $m$}
\label{mushrooms_quad_itm}
\end{center}
\vskip -0.2in
\end{figure}

\begin{figure}[t]
\vskip 0.2in
\begin{center}
\centerline{\includegraphics[width=0.99\columnwidth]{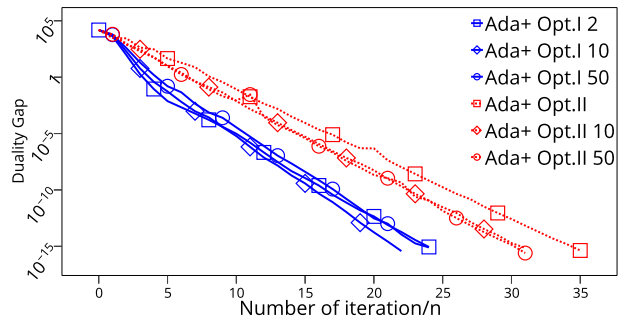}}
\caption{cov1 dataset $d$ = 54, $n$ = 581012, Quadratic loss with $L_2$ regularizer, comparison of different choices of the constant $m$}
\label{cov1_quad_itm}
\end{center}
\vskip -0.2in
\end{figure}

\clearpage

\begin{figure}[t]
\vskip 0.2in
\begin{center}
\centerline{\includegraphics[width=0.99\columnwidth]{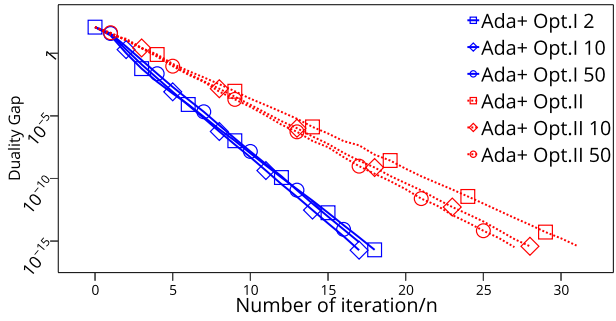}}
\caption{ijcnn1 dataset $d$ = 22, $n$ = 49990, Quadratic loss with $L_2$ regularizer, comparison of different choices of the constant $m$}
\label{ijcnn1_quad_itm}
\end{center}
\vskip -0.2in
\end{figure}

\begin{figure}[t]
\vskip 0.2in
\begin{center}
\centerline{\includegraphics[width=0.99\columnwidth]{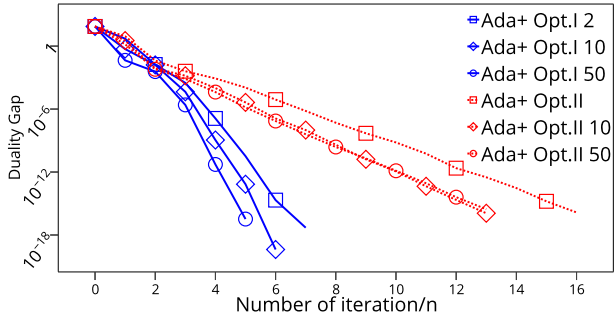}}
\caption{w8a dataset $d$ = 300, $n$ = 49749, Smooth Hinge loss with $L_2$ regularizer, comparison of different choices of the constant $m$}
\label{w8a_hinge_itm}
\end{center}
\vskip -0.2in
\end{figure}

\begin{figure}[t]
\vskip 0.2in
\begin{center}
\centerline{\includegraphics[width=0.99\columnwidth]{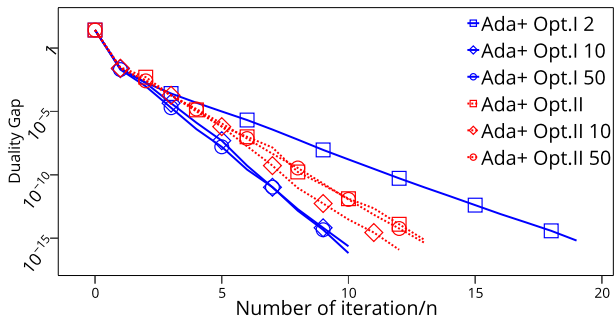}}
\caption{dorothea dataset $d$ = 100000, $n$ = 800, Smooth Hinge loss with $L_2$ regularizer, comparison of different choices of the constant $m$}
\label{dorothea_hinge_itm}
\end{center}
\vskip -0.2in
\end{figure}

\begin{figure}[t]
\vskip 0.2in
\begin{center}
\centerline{\includegraphics[width=0.99\columnwidth]{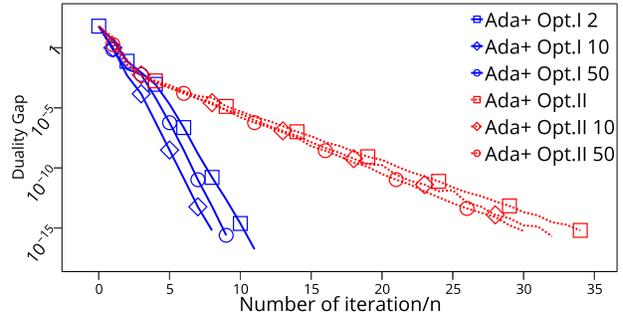}}
\caption{mushrooms dataset $d$ = 112, $n$ = 8124, Smooth Hinge loss with $L_2$ regularizer, comparison of different choices of the constant $m$}
\label{mushrooms_hinge_itm}
\end{center}
\vskip -0.2in
\end{figure}


\begin{figure}[ht]
\vskip 0.2in
\begin{center}
\centerline{\includegraphics[width=0.99\columnwidth]{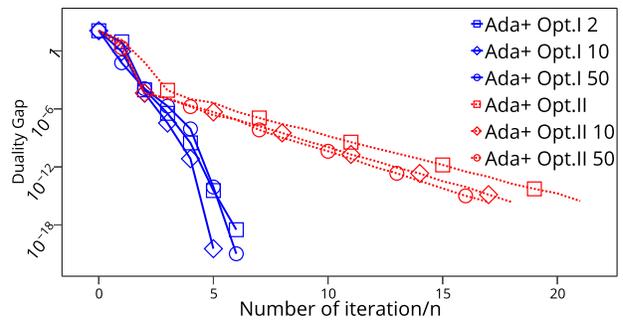}}
\caption{ijcnn1 dataset $d$ = 22, $n$ = 49990, Smooth Hinge loss with $L_2$ regularizer, comparison of different choices of the constant $m$}
\label{ijcnn1_hinge_itm}
\end{center}
\vskip -0.2in
\end{figure}  


 \clearpage
\bibliography{Paper_Adaptive}
\bibliographystyle{icml2015}

\clearpage
{\Large\textbf{Appendix}}
\section*{Proofs}\label{sec:proof}
We shall need the following inequality.
\begin{lemma} \label{lem:first}
Function $f : \R^n \rightarrow \R$ defined in \eqref{def:f} satisfies the following inequality:
\begin{equation}\label{eq:ff}
f(\alpha + h) \leq f(\alpha) + \langle\nabla f(\alpha),h\rangle + \frac{1}{2\lambda n^2}h^\top A^\top Ah, 
\end{equation}
holds for $\forall \alpha,h \in \R^n$.
\end{lemma}

\begin{proof}
Since $g$ is 1-strongly convex, $g^*$ is 1-smooth. Pick $\alpha, h \in \R^n$. Since, $f(\alpha) = \lambda g^*(\frac{1}{\lambda n} A\alpha)$, we have
\begin{align*}
&f(\alpha + h) = \lambda g^*(\frac{1}{\lambda n}A\alpha + \frac{1}{\lambda n}Ah) 
\\ &\leq \lambda \left( g^*(\frac{1}{\lambda n}A\alpha) + \langle\nabla g^* (\frac{1}{\lambda n}A \alpha), \frac{1}{\lambda n} Ah \rangle + \frac{1}{2} \| \frac{1}{\lambda n}Ah\|^2 \right)
\\ &= f(\alpha) + \langle \nabla f(\alpha), h \rangle + \frac{1}{2\lambda n^2}h^\top A^\top Ah. 
\\
\end{align*}

\end{proof}

\begin{proof}[Proof of Lemma~\ref{l:Theoutof}]
It can be easily checked that the following relations hold
\begin{align}\label{a:dfeggg}
&\nabla_i f(\alpha^t)=\frac{1}{n}A_{i}^\top w^t, \enspace \forall t\geq 0, \enspace i\in[n],\\
&g(w^t)+g^*(\bar \alpha^t)=\ve{w^t}{\bar \alpha^t} ,\enspace \forall t\geq 0, \label{a:gwtgalphat}
\end{align}
where $\{w^t, \alpha^t,\bar \alpha^t\}_{t\geq 0}$ is the output sequence of Algorithm~\ref{alg:A}.
 Let $t\geq 0$ and $\theta \in [0, \min_i p^t_i]$. 
For each $i\in [n]$, since $\phi_i$ is $1/\gamma$-smooth,  $\phi_i^*$ is $\gamma$-strongly convex and thus for arbitrary $s_i
\in [0,1]$,
\begin{align}\notag
 &\phi_{i}^*(-\alpha_{i}^{t}+ s_i \kappa_{i}^t)
\\&=\phi_i^*\left((1-s_i)(-\alpha_i^t)+
s_i \nabla\phi_i(A_i^\top w^t)\right) \notag
\\ & \leq (1-s_i)\phi_i^*(-\alpha_i^t)+s_i \phi_i^*(\nabla \phi_i(A_i^\top w^t)) \notag
\\& \qquad  -\frac{\gamma s_i(1-s_i)|\kappa_i^t|^2}{2}.\label{a-deerefdf}
\end{align}
We have:
\begin{align}\notag
& f(\alpha^{t+1})-f(\alpha^t)\\& \overset{\eqref{eq:ff}}{\leq}\notag
\ve{\nabla f(\alpha^t)}{\alpha^{t+1}-\alpha^t}\\&\qquad+\frac{1}{2\lambda n^2} \notag
\ve{\alpha^{t+1}-\alpha^t}{A^\top A(\alpha^{t+1}-\alpha^t)}\\
&=\nabla_i f(\alpha^t) \Delta\alpha_{i_t}^t+\frac{v_i}{2\lambda n^2}|\Delta\alpha_{i_t}^t|^2\notag
\\&\overset{\eqref{a:dfeggg}}{=}\frac{1}{n}A_{i_t}^\top w^t \Delta\alpha_{i_t}^t+\frac{v_i}{2\lambda n^2}|\Delta\alpha_{i_t}^t|^2
\label{a:falpha}
\end{align}
Thus,
\begin{align*}
& D(\alpha^{t+1})-D(\alpha^t) \\
& \overset{\eqref{a:falpha}}{\geq}
-\frac{1}{n}A_{i_t}^\top w^t \Delta\alpha_{i_t}^t-\frac{v_{i_t}}{2\lambda n^2}|\Delta\alpha_{i_t}^t|^2+ \frac{1}{n}\sum_{i=1}^n 
\phi_i^*(-\alpha_{i}^{t})\\
&\qquad \qquad-\frac{1}{n}\sum_{i=1}^n\phi_i^*(-\alpha_{i}^{t+1}) \\
&=
-\frac{1}{n}A_{i_t}^\top w^t \Delta\alpha_{i_t}^t-\frac{v_{i_t}}{2\lambda n^2}|\Delta\alpha_{i_t}^t|^2
+\frac{1}{n}\phi_{i_t}^*(-\alpha_{{i_t}}^{t})   \\ &\qquad  - 
\frac{1}{n}\phi_{i_t}^*(-\left(\alpha_{i_t}^{t}+\Delta\alpha_{i_t}^t\right))\\
&=\max_{\Delta\in \R}
-\frac{1}{n}A_{i_t}^\top w^t \Delta-\frac{v_{i_t}}{2\lambda n^2}|\Delta|^2
+\frac{1}{n}\phi_{i_t}^*(-\alpha_{{i_t}}^{t})   \\ &\qquad \enspace - 
\frac{1}{n}\phi_{i_t}^*(-\left(\alpha_{i_t}^{t}+\Delta\right)),
\end{align*}
where the last equality follows from the definition of $\Delta\alpha_{i_t}^t$ in Algorithm~\ref{alg:A}.
Then by letting $\Delta=-s_i\kappa_{i_t}^t$ for some arbitrary $s_i\in [0,1]$ we get:
\begin{align*}
& D(\alpha^{t+1})-D(\alpha^t) \\
&\geq 
\frac{s_i A_{i_t}^\top w^t \kappa_{i_t}^t}{n }-\frac{s_i^2 v_{i_t}|\kappa_{i_t}^t|^2}{2\lambda n^2 }
+\frac{1}{n}\phi_{i_t}^*(-\alpha_{{i_t}}^{t})   \\ &\qquad \enspace - 
\frac{1}{n}\phi_{i_t}^*(-\alpha_{i_t}^{t}+ s_i\kappa_{i_t}^t)
\\& \overset{\eqref{a-deerefdf}}{\geq}
\frac{s_i}{n}\left(\phi_{i_t}^*(-\alpha_{i_t}^t)-
 \phi_{i_t}^*(\nabla \phi_{i_t}(A_{i_t}^\top w^t)) + A_{{i_t}}^\top w^t \kappa_{i_t}^t\right) \notag
\\& \qquad -\frac{s_i^2 v_{i_t}|\kappa_{i_t}^t|^2}{2\lambda n^2 }
 +\frac{\gamma s_i( 1-s_i)|\kappa_{i_t}^t|^2}{2n}
.\end{align*}
By taking expectation with respect to $i_t$ we get:
\begin{align}\notag
&\Exp_t\left[D(\alpha^{t+1})-D(\alpha^t)\right] \\ \notag &\geq \sum_{i=1}^n \frac{ p^t_is_i}{n}
\left[\phi_{i}^*(-\alpha_{i}^t)-
 \phi_i^*(\nabla \phi_{i}(A_{i}^\top w^t)) +  A_{i}^\top w^t \kappa_{i}^t\right]\\ 
&\quad- \sum_{i=1}^n \frac{ p_i^t s_i^2|\kappa_i^t|^2(v_i+\lambda \gamma n)}{2\lambda n^2}+\sum_{i=1}^n
\frac{p_i^t \gamma s_i  |\kappa_i^t|^2}{2n}.\label{a:fderdfs}
\end{align}
 Set
\begin{equation}\label{dis:s_i}
s_i=\left\{\begin{array}{ll}
0,&  \quad  i\notin I_t\\
     \theta/ p_i^t,& \quad i\in I_t
    \end{array}\right. 
\end{equation}
Then $s_i \in [0,1]$ for each $i\in [n]$ and by plugging it into~\eqref{a:fderdfs} we get:
\begin{align*}
&\Exp_t\left[D(\alpha^{t+1})-D(\alpha^t)\right]\\
&\geq \frac{\theta}{n}\sum_{i\in I_t} 
\left[\phi_{i}^*(-\alpha_{i}^t)-
 \phi_i^*(\nabla \phi_{i}(A_{i}^\top w^t)) +  A_{i}^\top w^t \kappa_{i}^t\right]\\
& \quad-\frac{\theta}{2\lambda n^2}\sum_{i \in I_t}\left(  \frac{\theta(v_i+n\lambda\gamma)}{p_i^t} - n\lambda \gamma \right) |\kappa^t_i|^2 
\end{align*}
Finally note that:
\begin{align*}
 &P(w^t)-D(\alpha^t)\\
&=\frac{1}{n}\sum_{i=1}^n \left[ \phi_i(A_i^\top w^{t})+\phi_i^*(-\alpha_i^t)\right]+\lambda \left( g(w^t)+g^*(\bar \alpha^t)\right)
\\
&\overset{\eqref{a:gwtgalphat}}{=}\frac{1}{n}\sum_{i=1}^n \left[\phi_i^*(-\alpha_i^t)+\phi_i(A_i^\top w^{t})\right]+\frac{1}{n}\ve{w^t}{A\alpha^t}
\\&= \frac{1}{n}\sum_{i=1}^n \left[ \phi_i^*(-\alpha_i^t)+ A_i^\top w^t\nabla \phi_i(A_i^\top w^t) \right.\\ & \qquad \qquad\left.-
\phi_i^*(\nabla \phi_i(A_i^\top w^t))+ A_i^\top w^t \alpha_i^t\right]
\\& =\frac{1}{n} \sum_{i=1}^n  \left[\phi_i^*(-\alpha_i^t)-
\phi_i^*(\nabla \phi_i(A_i^\top w^t))+A_i^\top w^t \kappa_i^t\right]
\\& = \frac{1}{n} \sum_{i\in I_t}  \left[\phi_i^*(-\alpha_i^t)-
\phi_i^*(\nabla \phi_i(A_i^\top w^t))+A_i^\top w^t \kappa_i^t\right]
\end{align*}

\end{proof}

\begin{proof}[Proof of Lemma~\ref{lem:pstar}]

Note that~\eqref{eq:optprelaxed} is a standard constrained maximization problem, where everything independent of $p$ can be treated as a constant. We define the Lagrangian $$ L(p, \eta) = \theta(\kappa, p) - \eta(\sum_{i=1}^n p_i - 1)$$ and get the following optimality conditions:
\begin{align*}
&\frac{|\kappa^t_i|^2(v_i + n\lambda\gamma)}{p_i^2} = \frac{|\kappa^t_j|^2(v_j + n\lambda\gamma)}{p^2_j}, ~ \forall i,j \in [n]
\\ &\sum_{i=1}^n p_i = 1\\
& p_i \geq 0,\enspace \forall i\in [n],
\end{align*}
the solution of which is~\eqref{a:optipsdf}.
\end{proof}

\begin{proof}[Proof of Lemma~\ref{l:Theoutofquadratic}]
 Note that in the proof of Lemma~\ref{l:Theoutof}, 
the condition $\theta\in [0,\min_{i\in I_t} p_i^t]$
 is only needed to ensure that $s_i$ defined by~\eqref{dis:s_i} is in $[0,1]$ so that~\eqref{a-deerefdf} holds.
 If $\phi_i$ is quadratic function, then~\eqref{a-deerefdf} holds for arbitrary $s_i \in \R$.
Therefore in this case we only need $\theta$ to be positive and the same reasoning holds.
\end{proof}
\newpage

\end{document}

%% file: commands.tex
\usepackage{amsmath,amssymb, amsthm}
\usepackage{amsfonts}
\usepackage{layout}
\usepackage{amsmath}
\usepackage{amsthm}
\usepackage{amssymb}
\usepackage{url}
\usepackage{color}
\usepackage{graphicx}
\usepackage{verbatim}

\newcommand{\ve}[2]{\langle #1 ,  #2 \rangle}

\newcommand{\eqdef}{\stackrel{\text{def}}{=}}

\newcommand{\R}{\mathbb{R}}

\DeclareMathOperator{\mynnz}{{nnz}}

\DeclareMathOperator*{\argmax}{arg\,max}

\DeclareMathOperator{\Exp}{\mathbb{E}}           

\DeclareMathOperator{\dom}{dom}         

\theoremstyle{plain}
\newtheorem{theorem}{Theorem}
\newtheorem{proposition}[theorem]{Proposition}
\newtheorem{corollary}[theorem]{Corollary}
\newtheorem{lemma}[theorem]{Lemma}

\theoremstyle{definition}

\newtheorem{definition}[theorem]{Definition}

\makeatletter
\newcommand*{\rom}[1]{\expandafter\@slowromancap\romannumeral #1@}
\makeatother